\documentclass[12pt, twoside]{scrartcl}
\usepackage[automark, headsepline]{scrlayer-scrpage}

\usepackage{mathtools}  
\usepackage{amsmath}
\usepackage{amsthm}
\usepackage{amsxtra}

\usepackage{amssymb}

\usepackage[numbers]{natbib}
\usepackage{comment}

\usepackage[dvipsnames]{xcolor}

\usepackage{url}

\usepackage[english, psdextra]{hyperref}
\hypersetup{
  linkcolor = NavyBlue!80!Black,
  citecolor = Plum,
  urlcolor = Cerulean!70!Black,
  colorlinks = true
}

\newcommand{\wwedge}[1]{\sideset{}{^{#1}}\bigwedge}
\renewcommand{\Re}{\operatorname{Re}}

\newcommand{\Hom}{\operatorname{Hom}}

\newcommand{\R}{\mathbb{R}}
\newcommand{\Q}{\mathbb{Q}}
\newcommand{\Z}{\mathbb{Z}}

\newcommand{\C}{\mathbb{C}}
\newcommand{\F}{\mathbb{F}}
\newcommand{\W}{\mathbb{W}}

\newcommand*{\Ug}{\mathrm{U}}
\newcommand*{\Orth}{\mathrm{O}}
\newcommand*{\SL}{\mathrm{SL}}
\newcommand*{\SU}{\mathrm{SU}}
\newcommand*{\SO}{\mathrm{SO}}
\newcommand*{\Sp}{\mathrm{Sp}}

\newcommand*{\frakg}{\mathfrak{g}}
\newcommand*{\fraku}{\mathfrak{u}}
\newcommand*{\frakp}{\mathfrak{p}}
\newcommand*{\frakk}{\mathfrak{k}}
\newcommand*{\frako}{\mathfrak{o}}
\newcommand*{\Lie}{\operatorname{Lie}}

\newcommand{\h}{\mathbb{H}}
\newcommand{\Dom}{\mathbb{D}}
\newcommand{\Sym}{\operatorname{Sym}}

\DeclarePairedDelimiter{\hlfp}{(}{)}
\DeclarePairedDelimiter{\nvert}{\lvert}{\rvert}
\newcommand*{\hlf}[2]{\hlfp*{ #1, #2}} 
\newcommand*{\hlfempty}{\hlf{\cdot}{\cdot}}
\newcommand*{\blf}[2]{\hlf{#1}{#2}_{\R}}
\newcommand*{\blfempty}{\hlfempty_{\R}}
\newcommand*{\abs}[1]{\nvert*{#1}}

\DeclarePairedDelimiter{\sangl}{\langle}{\rangle}
\newcommand*{\sform}[2]{\sangl*{#1, #2}}

\newcommand*{\LL}{\langle\langle}
\newcommand*{\RR}{\rangle\rangle}

\newcommand*{\dual}{\sharp}
\newcommand*{\ebase}{\mathfrak{e}}

\newcommand*{\Fock}{S}
\newcommand*{\Schw}{\mathcal{S}}
\newcommand*{\Dhowe}{\mathcal{D}}
\newcommand*{\Da}{\mathcal{D}}
\newcommand*{\phikm}{\varphi_{KM}}
\newcommand*{\psikm}{\psi}
\newcommand*{\Psising}{\Psi}
\newcommand*{\psidc}{\psi'}

\newcommand*{\MfL}[2]{\mathrm{M}_{#1,#2}}
\newcommand*{\MfLw}[2]{\mathrm{M}_{#1,#2}^!}
\newcommand*{\CfL}[2]{\mathrm{S}_{#1,#2}}

\newcommand*{\HmfL}[2]{\mathrm{H}_{#1, #2}}
\newcommand*{\HmfpL}[2]{\mathrm{H}_{#1, #2}^+}
\newcommand*{\Amod}[1]{\mathrm{A}_{#1}^{mod}}
\newcommand*{\AmodL}[2]{\mathrm{A}_{#1, #2}^{mod}}
\newcommand*{\Aweak}[1]{\mathrm{A}_{#1}^{!}}
\newcommand*{\AweakL}[2]{\mathrm{A}_{#1, #2}^{!}}

\newcommand{\la}{\lambda}
\newcommand*{\cycl}[1]{\mathbb{D}(#1)}
\newcommand*{\cyclX}[1]{Z(#1)}
\newcommand*{\Green}{\mathcal{G}}

\DeclareMathOperator*{\CT}{CT}

\newcommand*{\pairregLm}[2]{\left( #1, #2 \right)^{reg}_{L^-}}

\newcommand*{\casimir}[1]{\mathrm{C}_{#1}}
\newcommand*{\CSL}{\casimir{\SL_2}}
\newcommand*{\CUpq}{\casimir{\Ug(p,q)}}
\newcommand*{\CO}{\casimir{\Orth(2p, 2q)}}

\newcommand*{\Fix}{\operatorname{Fix}}

\numberwithin{equation}{section}

\theoremstyle{plain}
\newtheorem{theorem}{Theorem}[section]
\newtheorem{corollary}[theorem]{Corollary}
\newtheorem{lemma}[theorem]{Lemma}
\newtheorem{proposition}[theorem]{Proposition}
\newtheorem*{theorem*}{Theorem}
\newtheorem*{proposition*}{Proposition}
\newtheorem*{corollary*}{Corollary}

\theoremstyle{definition}
\newtheorem{definition}[theorem]{Definition}
\newtheorem{remark}[theorem]{Remark}

\theoremstyle{remark}

\title{The construction of Green currents and singular theta lifts for unitary groups}
\author{Jens Funke, Eric Hofmann\footnote{Supported by DFG research fellowship HO 6123/1-1 in 2017/18}}
\date{}
\KOMAoptions{abstract=on}
\begin{document}

\maketitle

\begin{abstract}
With applications in the Kudla program in mind we employ singular theta lifts for the reductive dual pair $\Ug(p,q)\times \Ug(1,1)$ to construct two different kinds of Green forms for codimension $q$-cycles in Shimura varieties associated to unitary groups.
We establish an adjointness result between our singular theta lift and the Kudla-Millson lift. Further, we compare the two Greens forms and obtain modularity for the generating function of the difference of the two Green forms. Finally, we show that the Green forms obtained by the singular theta lift satisfy an eigenvalue equation for the Laplace operator and conclude that our Green forms coincide with the ones constructed by Oda and Tsuzuki by different means.
\end{abstract}

\tableofcontents

\section{Introduction}

Since its inception 20 years ago (\cite{K97}), the {\em Kudla program} has yielded many critical insights at the intersection of arithmetic geometry and automorphic forms. Roughly speaking, the Kudla program asserts the modularity of generating series of certain `special' cycles in (integral models of) orthogonal and unitary Shimura varieties when viewed as elements in the arithmetic Chow group, in particular as derivatives of Eisenstein series. For an overview, see e.g. \cite{K04}. Note that this program can be viewed as the considerable refinement and extension of the work of Kudla and Millson in the 1980's, see e.g. \cite{KM90}, which employed the theta correspondence and theta series to establish the modularity of the special cycles in the cohomology of locally symmetric spaces.

One important aspect in the Kudla program is the construction of appropriate Green currents for the complex points of the cycles which are then used to `upgrade' the cycles to define elements in the arithmetic Chow group associated to the underlying Shimura variety. 

For special (Heegner) divisors in Hermitian spaces associated to the orthogonal group $\Orth(p,2)$, Kudla \cite{K97,K03} constructed Green functions in terms of the exponential integral. Later Bruinier  and the first named author of this paper employed Borcherds' singular theta lift \cite{Bo98} and its extension to harmonic weak Maass forms to construct another Green function for the divisors  \cite{Br02,BrK03,BrF04}. In \cite{Br02} an adjointness result between the Borcherds lift and the Kudla-Millson lift was established, and in this context also the relationship in terms of the construction of these two Green functions was clarified. We note that the underlying structure for all these constructions is the dual reductive pair $\Orth(p,2) \times \SL_2(\R)$.

For the dual pair $\Ug(p,1) \times \Ug(1,1)$ the singular theta lift (for weakly holomorphic input) was first studied in detail by the second named author of this paper, see \cite{HDiss}. Its extension to harmonic weak Maass forms gives again Green functions for the special cycles which again are divisors. These have been utilized in \cite{BHY15,BHKRY17I, BHKRY17II}.

The difference between the two Green functions and its consequences in the context of the Kudla program are fairly subtle. This was studied and clarified by Ehlen and Sankaran \cite{ES16}. They show in the cases of $\Orth(p,2)$ respectively $\Ug(p,1)$ that the difference of the generating series can be viewed as a smooth modular form of weight $\tfrac{p}2 +1$ respectively $p+1$. 

Bruinier \cite{BrHilbert} considered the situation in the Hilbert modular case. In this case, Bruinier manages to circumvent the problem of the non-existence of the harmonic weak Maass forms to define a singular theta lift for `Whittaker forms' which then again gives rise to Green functions for the special divisors. 

For cycles of higher codimension much less has been known until recently. In Kudla's original work 
\cite{K97}, Liu \cite{Liu11}, and Bruinier and Yang \cite{BY18} star products are used to construct Green forms for cycles of higher codimension for $\Orth(p,2)$ and $\Ug(p,1)$. In recent groundbreaking work, Garcia and Sankaran \cite{GS} employed Quillen's theory of superconnections to construct Green forms in $\Orth(p,2)$ and $\Ug(p,q)$ in any codimension.

\bigskip

In this paper, we consider the construction of Green currents, in fact Green forms, for the dual pair $\Ug(p,q) \times \Ug(1,1)$. The associated Shimura varieties for $\Ug(p,q)$ are very attractive and natural objects to study. Furthermore, the cycles in question are no longer divisors but have codimension $q$. On the other hand, as $ \SU(1,1)$ is isomorphic to $\SL_2(\R)$ this case can be still approached via singular theta lifts of Borcherds type.

While this paper is certainly to a large extent written with applications in arithmetic geometry and the Kudla program in mind, we ignore this aspect in this paper completely and focus on the Archimedean side of the story. 

\bigskip

Let $V$ be an Hermitian space over an imaginary quadratic field of signature $(p,q)$. Then we can view the associated Hermitian domain $\Dom$ as the Grassmannian of negative $q$-planes in $V(\R)$. Let $L$ be an even lattice in $V$ and let $\Gamma$ be a finite-index subgroup of the stabilizer of $L$ in $\Ug(V(\R))$\footnote{In the main text we allow for a coset condition and work in the context of vector-valued modular forms. The results of this paper of course also hold in an appropriate adelic setting.}. We then define $X = \Gamma \backslash \Dom$ which gives a quasi-projective variety of dimension $pq$. 

To $x\in V$ with positive length we associate a subsymmetric space $\Dom(x) =\{ z \in \Dom; z \perp x\}$. Let $\Gamma_x$ be the stabilizer of $x$ in $\Gamma$, and we define the cycle $\cyclX{x}$ as the image of $\Gamma_x \backslash D_x$ in $X$. Note that the cycles have codimension $q$ in $X$ and arise from suitable embeddings $\Ug(p-1,q) \hookrightarrow \Ug(p,q)$. Finally, for $m>0$ we set
\[
\cyclX{m} =  \sum_{\substack{x \in L, \hlf{x}{x} = m\\ \mod \Gamma}} \cyclX{x} \in H^{q,q}(X)
\]
We let $\cyclX{0} = c_q$, the $q$-th Chern form on $\Dom$. Finally,  we set $\cyclX{m} = \emptyset$ for $m < 0$. 

The starting point for our considerations is the Kudla-Millson Schwartz form 
\[
\phikm \in  \bigl[\Schw(V)  \otimes \mathcal{A}^{q,q}(\Dom) \bigr]^G,
\]
which takes values in the closed differential $(q,q)$-forms in $D$. Under the action of the Weil representation of $\SO(2) \subset \SL_2(\R) \simeq \SU(1,1)$ it is an eigenfunction of weight $p+q$. Then the  associated theta series $\theta(z,\tau,\phikm)$ to $L$ ($\tau=u+iv \in \h$) is a (non-holomorphic) modular form of weight $p+q$ for a congruence subgroup $\Gamma' \subseteq \SL_2(\Z)$ with values in the closed differential $(q,q)$-forms in $X$. Furthermore, in cohomology we have 
\[
[\theta(z,\tau,\phikm)] = \sum_{m \geq 0} [\cyclX{m}] q^m.  \qquad \qquad (q =e^{2\pi i \tau})
\]
The key observation for our construction is 

\begin{theorem}
There exists a Schwartz form 
\[
\psikm \in  \bigl[\Schw(V)  \otimes \mathcal{A}^{q-1,q-1}(\Dom) \bigr]^G
\]
such that 
\begin{equation}\label{eq:Lddc}
 \omega(L)\, \phikm = dd^c \, \psikm.
 \end{equation}
Here $\omega(L)$ is the Weil-representation action of $L =  \tfrac12 \left( \begin{smallmatrix} \phantom{-} 1 & -i \\ -i & -1
\end{smallmatrix} \right) \in \mathfrak{sl}_2(\C) \simeq \mathfrak{su}(1,1)(\C)$ which corresponds to the Maass lowering operator $L=L_{p+q}$ for forms on the upper half plane, and $d$ and $d^c$ are the standard exterior derivatives acting on 
$\mathcal{A}^{\bullet}(\Dom)$. Furthermore, $\psikm$ has weight $p+q-2$ under the action of $\SO(2)$. 
\end{theorem}

Note that the solution to the equation $\omega(L) \varphi_{KM} = d \psi'$ was already constructed in \cite{KM90}, in fact, more generally for the dual pairs $\Orth(p,q) \times \Sp(n)$ and $\Ug(p,q) \times \Ug(n,n)$. In the same way our form $\psi$ can be used to solve the higher rank equations for $\Ug(p,q) \times \Ug(n,n)$. We explicitly construct $\psikm$ and establish its properties using the Fock model of the Weil representation, see Appendix \ref{sec:calc_Fock}. For convenience and future use we develop the formulas for the Weil representation much more generally for the dual pair $\Ug(p,q) \times \Ug(r,s)$. 

We then define the {\em Green form of Kudla type} by setting
\[
\Psising^0(x,z)  := - \int_{1}^\infty \psikm(\sqrt{t}x,z) e^{\pi t (x,x)} \frac{dt}{t}
\] 
for nonzero $x$ and then for $m \in \Q$ and $w >0$,
\[
\Xi^K(m,w)(z)  :=  \sum_{\substack{\lambda \in L,\lambda\neq 0  \\ 
\hlf{\lambda}{\la} = m}}\Psising^0(\sqrt{2w}\lambda, z),
\]
which defines a $(q-1,q-1)$-form on $X$ with singularities along the cycles $\cyclX{m}$ for $m>0$. For $m \leq 0$, the forms are smooth. 

The principle of this construction and its properties were already outlined in \cite{FK17} for $\Orth(p,q) \times \SL_2(\R)$ for the form $\psi'$ mentioned above and was also implicit in \cite{BrF04} for the Hermitian case $\Orth(p,2)$. Garcia and Sankaran \cite{GS} also follow these lines but use superconnections to solve \eqref{eq:Lddc}. We have not checked the details but it seems likely that for $n=1$ their form $\nu$ is equal to our form $\psi$. Garcia and Sankaran then succeed to construct Green forms for $n>1$ using a similar integral as above. 

On the other hand, we define a singular theta lift (of Borcherds type) using the theta series $\theta(z,\tau,\psikm)$ as integral kernel. Namely, for $f$, a harmonic Maass form of weight $k=2-p-q$, we set
\[
\Phi(z,f) := \int_{\Gamma' \backslash \h}^{reg} f(z) \theta(z,\tau,\psikm) d\mu(\tau).
\]
Here the regularization follows the by now standard procedure introduced by Harvey and Moore \cite{HM96} and Borcherds \cite{Bo98}. We then define for $m >0$ the {\em Green form of Bruinier type} by  
\[
\mathcal{G}^B(m)(z) := \Phi(z,F_m).
\]
Here $F_m(\tau)$ denotes the Hejhal Poincar\'e series of weight $k$ which has principal part $q^{-m}$ and `shadow' $\xi_k(F_m)=P_m$, the holomorphic Poincar\'e series for $\Gamma'$ of index $m$ and  weight $2-k=p+q$. Here $\xi_k= 2iv^{k} \overline{\tfrac{\partial}{\partial \bar{\tau}}} = v^{k-2}\overline{L_k}$ is the differential operator mapping forms of weight $k$ to weight $2-k$. For $m \leq 0$, we set $\mathcal{G}^B(m)(z)=0$. We show

\begin{theorem}
The forms $\Xi^K(m,w)$ and $\mathcal{G}^B(m)$ both define Green currents for the cycle $\cyclX{m}$. More precisely, as currents we have
\begin{align*}
dd^c[\Xi^K(m,w)] + (-i)^q\delta_{\cyclX{m}}&= [\varphi^0_{KM}(m,w)], \\
dd^c[\mathcal{G}^B(m)]  + (-i)^q \delta_{\cyclX{m}} &= [dd^c\Phi(F_m)].
\end{align*}
Here $\varphi^0_{KM}(m,w) = \sum_{\la \in L, (\la,\la)=m} \varphi_{KM}(\sqrt{2w}x)e^{2\pi mw}$.
\end{theorem}

The proof employs the same Lie-theoretic set-up as in \cite{BrF04} and \cite{FK17} for the orthogonal case. We first consider the analogous question for $\Psising^0(x)$, and as a consequence we obtain the Green property for $\Xi^K(m,w)$. We then show that $\mathcal{G}^B(m)$ has the same singularities as $\Xi^K(m,w)$, hence yielding the same residue. 

We can identify the term $dd^c\Phi(z,f)$ in the previous theorem explicitly as follows:

\begin{theorem} 
Let $f$ be a harmonic weak Maass form for $\Gamma'$ of weight $k=2-p-q$  with holomorphic constant term $a_0^+$, and let $\xi_k(f)$ be its shadow, a cusp form of weight $p+q$. Then
\[
dd^c \Phi(z,f)=\left( \Theta(\cdot, z, \phikm), \xi_k(f) \right)_{p+q}  + a^+(0,0)c_q
\]
as differential $(q,q)$-forms on $X$. Here $(\alpha, \beta)_\ell$ denotes the Petersson inner product in weight $\ell$. In particular, $dd^c \Phi(z,f)$ extends to a smooth closed $(q,q)$-form of moderate growth and $dd^c\Phi(z,f) = a^+(0,0)c_q$ for $f$ weakly holomorphic.
\end{theorem}

This can be viewed as an adjointness result between the Kudla-Millson lift and the singular theta lift associated to $\psi$. It is the analogue of the main result in \cite{BrF04}, and the proof follows along the same lines.

Following ideas of Ehlen and Sankaran \cite{ES16} we then compare the two Green forms in a different way. We show 

\begin{theorem} Assume $p+q>2$. Then for each $z \in \Dom$, the generating series
\[
F(\tau) = -\log(v)\psi(0)(z) -
\sum_{m\in\Q} \left( \Xi^K(m,v) - \Green^B(m) \right)(z)\, q^m 
\]
transforms like a smooth modular form of weight $p+q$. In addition, $F$ is orthogonal to cusp forms and satisfies $L_{p+q}F(\tau) = - \theta(\tau,\psikm)$.
 \end{theorem}

Finally, we construct for $m>0$ a different Green object $\Green_s^B(m)(z)$ depending on a complex parameter $s$. It is given essentially\footnote{Due to slightly different regularization process $\Green_{s_0}^B(m)(z)$ differs from $\Phi(z,F_m)$ by a smooth form.} as $\Phi(F_m(s),z)$, where $F_m(\tau,s)$ is the Hejhal Poincar\'e series of weight $k$ with complex parameter $s$ (at $s=s_0=1-k/2$ this is the weak Maass form $F_m$ introduced above). We show

\begin{theorem}
Let $\Delta$ be the Laplace operator acting on differential forms on $X$. Then
\[
\Delta \Green_s^B(m)= \left( (2s-1)^2 - (2s_0-1)^2\right) 
\Green^B_s(m).
\]
Furthermore, $\Green_s^B(m)$ agrees (up to a multiplicative constant) with the Green form constructed for $m>0$ by Oda and Tsuzuki \cite{OT09}. 
\end{theorem}

\bigskip

In view of applications in the Kudla program but also in its own right it will be interesting to consider suitable integrals of the singular theta lift $\Phi(z,f)$, say along the lines of \cite{K03} and \cite{BrK03}, and also to compute the Fourier-Jacobi expansion of the singular theta lift $\Phi(z,f)$ and to analyze the growth at the boundary components at suitable toroidal compactifications of $X$. We hope to come back to these questions in the near future. 

\bigskip

We thank Jan Bruinier, Stephan Ehlen, and Steve Kudla for very valuable discussions and suggestions. Funke thanks the Max Planck Institute for Mathematics in Bonn for multiple stays throughout the years. The initial considerations 
but also the final stages of this work were carried out there. Hofmann thanks the Department of Mathematical Sciences at Durham University for its hospitality during the academic year 2017/18. His stay was supported by a research fellowship ({\em Forschungsstipendium}) of the DFG.

\section{The unitary group}

\subsection{The unitary symmetric space}
	
We let $\left(V, \hlfempty\right)$ be a complex vector space of
dimension $m$ with a non-degenerate Hermitian form $\hlfempty$ of
signature $(p,q)$ with $p,q >0$. We assume that $\hlfempty$ is $\C$-linear in the
second and $\C$-antilinear in the first variable.  We pick standard
orthogonal basis elements $v_{\alpha}$ ($\alpha=1,\dots,p$) and
$v_{\mu}$ ($\mu=p+1,\dots, m$) of length $1$ and $-1$ respectively.\footnote{Throughout the paper we will follow \cite{KM90} and use `early' Greek letters for indices ranging from $1$ to $p$ and `late' for indices from $p+1$ to $m$.} We
let $z_{\alpha}$ and $z_{\mu}$ be the corresponding coordinate
functions so that
\[
  \hlf{x}{x} = \sum_{\alpha =1}^{p} \abs{z_{\alpha}}^2 - \sum_{\mu=
    p+1}^m \ \abs{z_{\mu}}^2,
\]
for
$x = \sum_{\alpha} z_{\alpha}v_{\alpha} + \sum_{\mu} z_{\mu} v_{\mu}
\in V$. The choice of basis also gives a decomposition $V=V_+ \oplus V_-$ into definite subspaces. We let $G=\Ug(V)$ be the unitary group of $V$, and let
$\Dom = G/K$ be the associated symmetric space of complex dimension $pq$. Here 
$K \simeq \Ug(p) \times \Ug(q)$ is the maximal 
compact subgroup corresponding to the basis of $V$ chosen above. 
We realize the symmetric space as the Grassmannian of negative definite 
$q$-planes in $V$:
\[
\Dom \simeq \left\{ z \subset V:\, \operatorname{dim}(z) = q,\; \hlf{\cdot}{\cdot}\vert_z < 0  \right\}. 
\]
Given $z\in \Dom$, the standard majorant $\hlf{x}{x}_z$ is given by
\[
\hlf{x}{x}_z = \hlf{x_{z^\perp}}{x_{z^\perp}} - \hlf{x_z}{x_z},
\]
where $x = x_z + x_{z^\perp}$ using the orthogonal decomposition $V = z\oplus z^\perp$. We also set 
\[
R(x,z) \vcentcolon= - \hlf{x_z}{x_z}.
\]
Note that $R(x,z) \geq 0$ with $R(x,z) = 0$ if and only if $x\in z^\perp$. 
When $x$ has positive norm, let $\cycl{x}$ denote the codimension $q$ sub-Grassmannian  
\[
\cycl{x} \vcentcolon = \{ z \in \Dom:\, z\perp x\} = 
\{ z \in \Dom:\, R(x,z) = 0\}.
\]
Also, for convenience, if $x$ is non-positive, set $\cycl{x} = \emptyset$.

Let $L\subset V$ be an even Hermitian lattice, i.e., a projective module over 
the ring of integers $\mathcal{O}_{\F}$ of an imaginary quadratic number field 
$\F$, on which the restriction of $\hlfempty$ is $\mathcal{O}_{\F}$-valued. We 
fix an embedding of $\F$ into $\C$. Denote by $\mathcal{D}_{\F}^{-1}$ the 
inverse different ideal of $\F$. The dual lattice $L^\dual$ is given by 
\[
L^\dual 
= \{ x\in V: \hlf{x}{\lambda}\in \mathcal{D}_{\F}^{-1},\, \forall \lambda \in L\}
= \{ x\in V: \operatorname{trace}_{\F/\Q}\hlf{x}{\lambda} \in \Z,\, \forall \lambda \in L\}.
\]
Note that $L\subset L^\dual$. 
The quotient $L^\dual/L$ is called the discriminant group of $L$. 

For $m\in \Q$  and $h\in L^\dual/L$, we define the special cycle $\cycl{m,h}$ in $\Dom$ by  
\[
\cycl{m, h} = \sum_{\substack{\lambda \in L + h\\ \hlf{\lambda}{\lambda} = m}} \cycl{\lambda}.  
\]
Note that $\cycl{m, h}$ is locally finite. We let $\Gamma = \Fix(L^{\#}/L) \subset G$ and write 
\[
X = 
\Gamma\backslash \Dom
\]
for the resulting quasi-projective variety. Further, we let $\cyclX{x}$ respectively $\cyclX{m,h}$ be the image of $\cycl{\lambda}$ respectively $\cycl{m,h}$ in $X$.

\subsection{The unitary Lie algebra}\label{sec:uLalg}
We let $\frakg_0 = \fraku(V)$ be the Lie algebra of $G$.  We define
the $\R$-linear surjective map
\[
  \phi_V: {\wwedge{2}}_\R V \longrightarrow \fraku(V)
\]
by
\[
  \phi_V(v \wedge \tilde{v})(x) =\hlf{v}{x}\tilde{v} - \hlf{\tilde{v}}{x}v
\]
Note that we have
\[
  \phi_V(iv \wedge \tilde{v}) = \phi_V(v \wedge -i\tilde{v}).
\]
In the following we will abuse notation and drop $\phi_V$ and just
write $v \wedge \tilde{v} \in \mathfrak{u}(V)$. Note that in this way we realize $\mathfrak{u}(V)$ as a quotient of ${\wwedge{2}}_\R V$ by the relation $iv \wedge \tilde{v}+ v \wedge i\tilde{v}=0$. We have
\[
  \frakg_0 = \operatorname{span}_{\R}\{ v_r \wedge v_s, iv_r \wedge
  v_s \}.
\]
We put
\[
  X_{rs} = v_r \wedge v_s \qquad \qquad \text{and} \qquad \qquad
  Y_{rs} = iv_r \wedge v_s.
\]
In the Cartan decomposition $\frakg_0 = \frakk_0 \oplus \frakp_0$ with
$\frakk_0 = \Lie(K) = \fraku(p) \times \fraku(q)$, we note that
\[
  \mathfrak{p}_0 = \operatorname{span}_{\R}\{ X_{\alpha\mu},
  Y_{\alpha\mu}; 1\leq \alpha \leq p, p+1 \leq \mu \leq m\}.
\]
We let $\{\omega_{\alpha \mu},\omega_{\beta\nu}'\}$ be the corresponding dual basis for $\frakp_0^{\ast}$.
Furthermore, the natural complex structure on $\frakp_0$ 
is given by $ X_{\alpha \mu} \mapsto Y_{\alpha\mu}$;
$ Y_{\alpha \mu} \mapsto -X_{\alpha\mu}$.
	
We let $\frakg = \frakg_0 \otimes \C$ be the complexification of
$\frakg_0$, which we view as a right $\C$-vector space. We define
\[
  Z_{rs}' = \frac12 (X_{rs} -Y_{rs}i) \qquad \qquad \text{and} \qquad
  \qquad {Z}_{rs}'' = \frac12 (X_{rs} +Y_{rs}i).
\]
Note that $Z_{rs}'' = -Z_{sr}'$. In the Harish-Chandra decomposition
\[
  \frakg = \frakk \oplus \frakp^+ \oplus \frakp^-,
\]
we see that
\begin{align*}
  \frakk = \operatorname{span}_{\C} \{ Z_{\alpha\beta}',  Z_{\mu\nu}'\}, \qquad
  \frakp^+ = \operatorname{span}_{\C}\{ Z_{\alpha\mu}' \}, \qquad
  \frakp^- = \operatorname{span}_{\C} \{ {Z}_{\alpha\mu}''  \}.
\end{align*}
We let $\{\xi_{\alpha \mu}'\}$ and $\{\xi_{\alpha \mu}''\}$ be the corresponding dual basis of $ \frakp^+$ and $ \frakp^-$. 

\bigskip

We let $V_\C = V \otimes_\R \C$. We view $V_\C$ as a {\emph{right}} complex
vector space of dimension $2m$ and hence write $vi$ for $v \otimes i$. Note that $iv$ (internal multiplication of the left $\C$-vector space $V$) is not equal to $vi$. We decompose $V_\C = V' \oplus V''$
into the $+i$ and $-i$ eigenspaces under left multiplication by
$i$. The maps
\[
  v \longmapsto v - ivi \qquad \text{and} \qquad v \longmapsto v + ivi
\]
realize a $\C$-linear isomorphism of (the left $\C$-vector space) $V$
with (the right $\C$-vector space) $V'$ and a $\C$-anti-linear
isomorphism with $V$ and $V''$. Hence we can view
$V'' \simeq V^{\ast}$ as $\C$-vector spaces. We denote the natural
bases of $V'$ and $V''$ by
\[
  v_r':= v_r -iv_ri \qquad \text{and} \qquad v_r'':=v_r+iv_ri,
\]
respectively. Furthermore, we obtain decompositions $V'=V'_+ \oplus V'_-$ and $V''=V''_+ \oplus V''_-$ in the natural way. We have
\[
  Z'_{rs}(v_t') = -(v_s,v_t) v_r' \qquad \text{and} \qquad
  Z'_{rs}(v_t'') = (v_r,v_t) v_s'',
\]
and we note that this realizes the isomorphism
$\frakg \simeq \mathfrak{gl}_m(\C)$ by the action of $\mathfrak{g}$ on
$V'$. More precisely, we obtain
\[
\mathfrak{k} \simeq \Hom(V_+',V_+') \oplus  \Hom(V_-',V_-'),  \qquad \mathfrak{p}^+ \simeq \Hom(V_-',V_+'), \qquad \mathfrak{p^-} \simeq \Hom(V'_+,V'_-).
\]
Correspondingly, the action of $\frakg$ on $V''$ realizes the dual of the standard representation of $\frakg$. 

\bigskip

We write $V_\R$ for $V$ considered as a real quadratic space together
with the quadratic form $\blfempty = \Re \hlfempty = \tfrac12 \operatorname{trace}_{\C/\R}\hlfempty$\footnote{This works better for our purposes.}. Then
$\{v_{\alpha}, iv_{\alpha}, v_\mu,iv_\mu\}$ forms an orthogonal basis
of $V_\R$.
We let $\mathfrak{o}_{V_\R}$ be the Lie algebra of the orthogonal
group $\Orth(V_\R)$. We now have the isomorphism
\[
  \phi_{V_\R}: {\wwedge{2}} V_\R \simeq \frako(V_\R)
\]
given by
\[
  \phi_{V_\R}(v \wedge \tilde{v})(x) = \blf{v}{x} \tilde{v} - \blf{\tilde{v}}{x}v.
\]
We let $\iota:\frakg_0= \fraku(V) \mapsto \frako(V_\R)$ be the natural
embedding. We easily see
\[
  \iota(\phi_V(v\wedge \tilde{v})) = \phi_{V_\R}(v\wedge \tilde{v}) +
  \phi_{V_\R}(iv\wedge i\tilde{v}).
\]
Note this realizes $\fraku(V)$ as the subspace of ${\wwedge{2}}_\R V$ which is fixed by (left)-multiplication with $i$ in both factors.

\section{Schwartz forms}\label{sec:Schwartz}

\subsection{Weil representation} 

Let $\Schw(V)$ be the Schwartz space of $V$. Associated to an additive character $\psi$ of $\R$ we consider the Weil representation $(\omega,\psi)$ for the dual reductive pair $U(1,1) \times \Ug(V)$, acting in the Schr\"{o}dinger model on $\Schw(V)$. Recall that all such characters are given by $\psi_\alpha(t)=e(\alpha t)$ with $\alpha \in \R$. Here $e(t)=e^{2\pi i t}$ as usual. The setup of the Weil representation in the polynomial Fock model is explained in detail in the Appendix \ref{sec:focksetup}. We note that for any Schwartz function $\phi(x)$, the unitary group $\Ug(V)$ acts linearly,
$\omega(g) \phi(x) = \phi(g^{-1}x)$. For 
matrices in $\SL_2(\R) \simeq 
\SU(1,1)$ it is given as follows, see eg Shintani \cite{Shin75}:
\begin{align*}
\omega\left(\begin{psmallmatrix}
 1 & b \\ 0 & 1
\end{psmallmatrix}\right)\phi(x) & = \psi_{\alpha}(\tfrac12 b\hlf{x}{x})\phi(x),\\
\omega\left(\begin{psmallmatrix}
 a & 0  \\ 0  & a^{-1}
\end{psmallmatrix}\right)\phi(x) & 
= a^m \phi(ax), \\
\omega\left(\begin{psmallmatrix}
 0 & -1 \\ 1  & 0
\end{psmallmatrix}\right)\phi(x) & 
=  i^{q-p} \widehat{\phi}(x),
\end{align*}
where $\widehat\phi(x) = \alpha^{m} \int_V \phi(y)\psi_{\alpha}(-(y,x)_\R)dy$ 
denotes the Fourier transform of $\phi(x)$. Here we identify $V$ with $\R^{2m}$ and $dy$ denotes the usual Lebesgue measure.   
Note that for $\alpha>0$ all representations $(\omega,\psi_{\alpha})$ are isomorphic. Explicitly, an intertwiner of $(\omega,\psi_{1})$ with $(\omega,\psi_{\alpha})$ can be given by $\phi(x)\mapsto \phi(\sqrt{\alpha}x)$. From now on we will take $\alpha =1$ so that the additive character is given by $ t \mapsto e(2\pi it)$. We say $\phi$ has weight $r\in\Z$  
if $\omega(k_\theta')\phi = e^{r i \theta}\phi$ for $k_\theta' = 
\begin{psmallmatrix}\cos\theta & \sin\theta \\ - \sin\theta & \cos\theta 
\end{psmallmatrix}$ in $K' = \Ug(1) \simeq \mathrm{SO}(2)$, the maximal compact subgroup of 
$\SL_2(\R) \simeq \SU(1,1)$\footnote{We could also work with $\Ug(1)\times \Ug(1)$ inside $\Ug(1,1)$, but we won't need this for our purposes.}. Note the standard Gaussian
\[
\varphi_0(x,z):= e^{-\pi\hlf{x}{x}_z} 
\]
has weight $p-q$\footnote{This coincides with the normalizations given in \cite{KM86}, see Section~\ref{KM-section}. This will cause some complications when defining theta series later where it would have been more convenient to pick the Weil representation for $\psi_{\alpha}$ with $\alpha=2$. However, we think it is more important to stick with the Kudla-Millson conventions fo the construction of the Schwartz forms.}. For $\tau = u + iv \in \mathbb{H}$, let $g_\tau'$  be any element of 
$\SL_2(\R)$, mapping $i$ to $\tau$. Then for $\phi$ of weight $r$ we set
\begin{equation}\label{phixtau}
\phi(x, \tau)  \vcentcolon =  v^{-\frac{r}{2}} 
\omega(g_\tau')\phi(x)  =  v^{-\frac{r}{2} + \frac{p + q}{2}} \phi^0(\sqrt{v}x)e^{\pi i 
\hlf{x}{x}\tau}. 
\end{equation}
Here we set 
\[
\phi^0(x)= e^{\pi (x,x)} \phi(x),
\]
which will be convenient throughout.

\subsection{The Kudla-Millson form \texorpdfstring{$\phikm$}{}}\label{KM-section}

We now consider the complex $\left[ \Schw(V) \otimes \mathcal{A}^{\bullet}(\Dom) \right]^G$ of $G$-invariant Schwartz functions on $V$ with values in the differential forms on $\Dom$. Note that evaluation at the base point $z_0$ yields an isomorphism 
\[
  \left[ \Schw(V) \otimes \mathcal{A}^{\bullet}(\Dom) \right]^G \simeq \bigl[
  \Schw(V) \otimes \wwedge{\bullet}(\frakp^*)\bigr]^K. 
\]
We use the same symbol for corresponding objects. Note
\[
  \varphi_0(x,z)  \in \left[ \Schw(V) \otimes
    \mathcal{A}^0(\Dom) \right]^G,
\]
and evaluation at the base point gives $ \varphi_0(x) = \varphi_0(x,z_0) =  e^{-\pi\sum_{i=1}^m \abs{z_i}^2} \in \Schw(V)^K$.

Following \citep[][Proposition 5.2]{KM86} and \citep[][Section 5]{KM90}, we 
define the differential operator
\[
  \Dhowe = \frac{1}{2^{2q}} \prod_{\mu = p+1}^m \left\lbrace
    \sum_{\alpha=1}^p \left(\bar{z}_\alpha -
      \frac{1}{\pi}\frac{\partial}{\partial z_\alpha} \right) \otimes
    A'_{\alpha\mu} \right\rbrace = \frac{1}{2^{2q}} \prod_{\mu = p+1}^m
  \left\lbrace \sum_{\alpha=1}^p \Da_\alpha \otimes A'_{\alpha\mu}
  \right\rbrace,
\]
where
$\tfrac{\partial}{\partial z_{\alpha}} = \tfrac12 \left(
  \frac{\partial}{\partial x_{\alpha}} -i \tfrac{\partial}{\partial
    y_{\alpha}} \right)$, and $A'_{\alpha\mu}$ denotes the left
multiplication with $\xi'_{\alpha\mu}$. Also, we have set $\Da_\alpha \vcentcolon = \left( \bar{z}_\alpha -
  \frac{1}{\pi}\frac{\partial}{\partial z_\alpha}\right)$.  

Following Kudla and Millson \cite{KM86}, we then define
\[
  \phikm \vcentcolon= \Dhowe \overline{\Dhowe} \varphi_0 
\in \bigl[\Schw(V)   \otimes \wwedge{q,q} \, (\frakp^{\ast}) \bigr]^K \simeq 
 \bigl[\Schw(V)   \otimes \mathcal{A}^{q,q}(\Dom) \bigr]^G.
\]
Thus, using multi-index notation with $\underline\alpha = \{\alpha_1, \dotsc, \alpha_q\}$ and $\underline{\beta} = \{\beta_1, \dotsc, \beta_q\}$,
  \begin{align*}
    \phikm =
        \frac{1}{2^{2q} } \sum_{\underline{\alpha}, 
      \underline{\beta}} 
      \Da_{\underline\alpha} \overline{D}_{\underline{\beta}}\,\varphi_0 
      \otimes \Omega_q(\underline{\alpha};\underline{\beta}),
  \end{align*}
 where $\Da_{\underline\alpha} = \prod_{j=1}^q  \Da_{\alpha_j}$ and 
\begin{align*}
\Omega_q(\underline{\alpha};\underline{\beta}) &= \xi'_{\alpha_1 p+1}
    \wedge \dotsb \wedge \xi'_{\alpha_q p+q} \wedge \xi''_{{\beta_1}
      p+1} \dotsb \wedge \xi''_{{\beta_q} p+q} \\
      &=  (-1)^{q(q-1)/2} \xi'_{\alpha_1 p+1} \wedge  \xi''_{{\beta_1}
      p+1} \wedge \dotsb \wedge \xi'_{\alpha_q p+q} \wedge \xi''_{{\beta_q} p+q}.
\end{align*}

The properties of the Schwartz form $\phikm$ are summarized in the following 
theorem.

\begin{theorem}[Kudla-Millson]\label{KM-properties}
The Schwarz form $\phikm$ has the following properties:
\begin{enumerate}
\item $\phikm$ is an eigenfunction of weight $p+q$ under the operation of 
$K'$ \citep[see][]{KM86}.
\item As a differential form, $\phikm(x,z)$ is closed for every $x\in V$ 
\citep[see][Section 4]{KM86}. 
\item The Thom Lemma holds for $\phikm$ \citep[see][Theorem 4.1]{KM87}, i.e.,
\[
\int_{\Gamma_x\backslash \Dom} \eta \wedge \phikm(x) = i^{-q} \left( \int_{\Gamma_x\backslash 
\cycl{x}} \eta \right) e^{-\pi (x,x)}
\]
for any compactly supported closed differential $2(p-1)q$ form $\eta$ on $\Gamma_x\backslash \Dom$.
\end{enumerate}
\end{theorem}

\subsection{The Schwartz form \texorpdfstring{$\psikm$}{of Millson-type}}\label{sec:Millson-form}

We define another Schwartz form $\psikm$ by setting
\[
 \psikm\vcentcolon = \frac{2 i (-1)^{q-1}}{2^{2(q-1)}} \sum_{\substack{\underline\alpha = \{\alpha_1, \dotsc, \alpha_{q-1}\} \\\underline{\beta} = \{\beta_1, \dotsc, \beta_{q-1}\}}} \Da_{\underline\alpha} \overline{D}_{\underline{\beta}}\,\varphi_0 
      \otimes \Omega_{q-1}(\underline{\alpha};\underline{\beta})
 \]
where
\begin{align*}
& {\Omega_{q-1}}(\underline{\alpha};\underline{\beta}) \\
 &=  (-1)^{q(q-1)/2}  \sum_{j=1}^q\xi'_{\alpha_1 p+1} \wedge  \xi''_{{\beta_1}
      p+1} \wedge \dotsb \wedge \widehat{ \xi'_{\cdot p+j} \wedge  \xi''_{\cdot 
      p+j}} \dotsb  \wedge \xi'_{\alpha_{q-1} p+q} \wedge \xi''_{{\beta_{q-1}} p+q}.
\end{align*}

In Appendix \ref{sec:calc_Fock} we will employ the Fock model of the Weil representation to show
\begin{proposition}\label{psi-properties}
\begin{itemize}
\item[(i)]
The Schwartz form $\psikm$ is invariant under the operation of $K$, that is, 
\[
\psikm \in 
\bigl[\Schw(V)   \otimes \wwedge{q-1,q-1} \, (\frakp^{\ast}) \bigr]^K \simeq 
 \bigl[\Schw(V)   \otimes \mathcal{A}^{q-1,q-1}(\Dom) \bigr]^G.
\]

\item[(ii)] 
The Schwartz form $\psi$ is an eigenfunction of weight $p+q-2$ under the operation of $K'$. 

\end{itemize}

\end{proposition}

The main property linking $\phikm$ and $\psikm$ is the following.
\begin{theorem}\label{prop:props_psikm}
 Let $d = \frac12\left(\partial + \bar\partial\right)$ and $d^c =  \frac{1}{4\pi i}\left(\partial - 
\bar\partial\right)$ be the standard exterior derivatives acting on 
$\mathcal{A}^{\bullet}(\Dom)$, and let 
 $L_\kappa = - 2iv^2\frac{\partial}{\partial\bar\tau}$ be the Maass lowering 
 operator of weight $\kappa$ acting on functions on the upper half plane. 
  Then
 \[
   L_{p+q}\,\phikm(x,\tau, z) = dd^c \, \psikm(x,\tau, z).
  \]
This implies 
  \[
	v\frac{\partial}{\partial v}\phikm^0(\sqrt{v}x,z) = dd^c \psi^0(\sqrt{v}x,z).
  \]
\end{theorem}
\begin{proof}
This is carried out in Appendix \ref{sec:calc_Fock}, again using the Fock model.
\end{proof}

In order to derive a more explicit description of the Schwartz form $\psi$, 
when evaluated at the base point $z_0$, we examine the properties of the 
differential operators $\Da_\alpha$ and $\bar{\Da}_{\alpha}$ for $\alpha\in \{ 
1, 
\dotsc, p\}$. First, we note that all the differential operators commute, i.e.\
$\Da_\alpha \Da_\beta = \Da_\beta \Da_\alpha$, $\bar\Da_\alpha\bar\Da_\beta = \bar\Da_\beta\bar\Da_\alpha$ and $\Da_\alpha\bar\Da_\beta = \bar\Da_\alpha \Da_\beta$ for all $\alpha, \beta \in \{1, \dotsc, p\}$.

Further, by direct calculation, we get
\[
\Da_\alpha \varphi_0 = 2 \bar{z}_\alpha \varphi_0,
\quad \bar{\Da}_\alpha \varphi_0 = 2 {z}_\alpha 
\varphi_0\quad\text{and}\quad
 \Da_\alpha\bar{\Da}_\alpha \varphi_0 = \left(4\abs{z_\alpha}^2  - 
 \tfrac{2}{\pi} 
 \right)\varphi_0. 
\]
In fact (see e.g., \citep[][p.\ 303 (6.41)]{KM87}),
\begin{equation}\label{eq:KM_laguerre}
  \Da_\alpha^k \bar\Da_\alpha^k \varphi_0= \left( \Da_\alpha \bar\Da_\alpha\right)^k \varphi_o=
  \left(\frac{1}{\pi} \right)^k 2^k k! L_{k}\left(2\pi \abs{z_\alpha}^2\right) \varphi_0,
\end{equation}
where $L_k(t)= \tfrac{e^t }{k!} \left( \tfrac{d}{dt}\right)^k \left(e^{-t} t^k\right)$ is the $k$-the Laguerre polynomial. More generally, we get  
\begin{equation}\label{eq:Dops}
\Da_\alpha^{l} \bar{\Da}_\alpha^{k} \varphi_0 =
2^{k} \sum_{m=0}^l \binom{l}{m} \sum_{n=0}^{\min(m,k)} \bar{z}_\alpha^{l-n} z_\alpha^{k-n} \binom{m}{n}
\frac{k!}{(k-n)!} \left(\frac{-1}{\pi}\right)^n \, \varphi_0.
\end{equation}
Hence the Schwartz form $\psikm$ can be expressed using (in general non-homogeneous) polynomials $P_{\underline{\alpha}, \underline{{\beta}}}^{2q-2} 
\in \mathcal{P}(V)$ as
follows:
\begin{align}\label{psikm-formula}
\psikm(x, z_0) &=
 \frac{2 i (-1)^{q-1}}{2^{2(q-1)}} 
\sum_{\underline{\alpha}, \underline{{\beta}}} 
P_{\underline{\alpha}, \underline{{\beta}}}^{2q-2}\bigl( x \bigr) \varphi_0(x)
\otimes \Omega_{q-1}(\underline{\alpha};\underline{\beta}),\\
\psikm^0(x, z_0)
&= 
 \frac{2 i (-1)^{q-1}}{2^{2(q-1)}} 
\sum_{\underline{\alpha}, \underline{\beta}}
P^{2q-2}_{\underline{\alpha}, \underline{\beta}}(x)\, e^{-2\pi R(x,z_0)} \otimes \Omega_{q-1}(\underline{\alpha};\underline{\beta}).
\end{align}
The following lemma is easily obtained.

 \begin{lemma}\label{lemma:polyexpl}
For any pair of multi-indices $\underline{\alpha}$, $\underline{\beta}\in \{1, \dotsc, p\}^{q-1}$, the
attached polynomial $P_{\underline{\alpha}, \underline{{\beta}}}^{2q-2}(x)$ has the following properties: 
\begin{enumerate}
\item It has degree $2q-2$ and depends only on $V_+$. 
\item The leading term is given by 
\[
2^{2(q-1)} \prod_{l = 1}^{q-1} \bar{z}_{\alpha_l} \prod_{k = 1}^{q-1} z_{\beta_k}.
\]
\item All monomials occurring in $P_{\underline{\alpha}, \underline{{\beta}}}^{2q-2}(x)$ have even degree. 
\item The constant term  
is non-zero if and only if for every $\alpha \in \{1, \dotsc, p\}$ the multiplicity of $\alpha$ in the multi-indices $\underline\alpha$ and $\underline{\beta}$ is the same. 
In which case,  $P_{\underline{\alpha}, \underline{{\beta}}}^{2q-2}(x)$ is a product of Laguerre functions, and the constant term is given by
\[
2^{q-1}\left(\frac{-1}{\pi}\right)^{q-1}\prod_{\alpha\in\underline{\alpha}} m(\alpha)!,
\]
where $m(\alpha)$ is the multiplicity of $\alpha$.
\end{enumerate}
\end{lemma}
In particular, the situation in part 4 of the lemma  occurs when $x = z_\alpha v_\alpha$, and only the terms with $\underline{\alpha} 
= \underline{\beta} = (\alpha, \alpha, \dotsc, \alpha)$ are non-zero.

We write
\begin{equation}\label{psi-poly}
P_\psi(x, z_0)= \frac{2 i (-1)^{q-1}}{2^{2(q-1)}} 
\sum_{\underline{\alpha}, \underline{{\beta}}} 
P_{\underline{\alpha}, \underline{{\beta}}}^{2q-2}\bigl( x \bigr) 
\otimes \Omega_{q-1}(\underline{\alpha};\underline{\beta})
\end{equation}
for the polynomial part of $\psi$. Furthermore, it will be convenient to write $P_{\underline{\alpha}, \underline{{\beta}}}^{2q-2}$ as a sum of its homogeneous components,
\[
P_{\underline{\alpha}, \underline{{\beta}}}^{2q-2}(x) 
= \sum_{\ell = 0}^{q-1} P_{\underline{\alpha}, \underline{{\beta}}; 2\ell}^{2q-2}(x),
\]
with $2\ell$ the respective weight. Note that  $P_{\underline{\alpha}, 
\underline{{\beta}}; 2\ell}^{2q-2}(w x) = 
\abs{w}^{2\ell}P_{\underline{\alpha}, \underline{{\beta}}; 2\ell}^{2q-2}(x)$ 
for any $w\in \C$. 
\begin{remark}
We mention that besides \eqref{eq:Dops} the polynomials $P_{\underline{\alpha}, \underline{{\beta}}}^{2q-2}(x)$ can also be expressed using derivatives of Laguerre functions by \eqref{eq:KM_laguerre} or, alternatively through Hermite functions in the real and imaginary parts of the $z_\alpha$'s as indeterminates. 
\end{remark}

\section{A singular Schwartz form}\label{sec:Psising}

Analogously to Kudla \cite{K04} for $\Orth(p,2)$, we define for $x\neq 0$ the singular Schwartz form
\begin{equation}\label{def:Kudla-xi}
\Psising^0(x,z) \vcentcolon = - \int_{1}^\infty \psikm^0(\sqrt{t}x,z) \frac{dt}{t}.
\end{equation}
The form $\Psising^0$ has its singularities where $R(x,z) = 0$, i.e., precisely along the cycles $\cycl{x}$. Thus, in particular,  $\Psising^0(x,z)$ is smooth for $\hlf{x}{x}\leq 0$. We also set
\[
\Psising(x,z) = \Psising^0(x,z) e^{-\pi\hlf{x}{x}}. 
\]

Recall the definition of the incomplete $\Gamma$-function, $\Gamma(s,a) = 
\int_a^\infty t^{s-1} e^{-t} dt$. The following lemma is obtained by a  
straightforward calculation.
\begin{lemma}\label{lemma:psitildeexpl}
At the base point $z=z_0$, the singular Schwartz form $\Psising^0$ is given by 
\begin{multline*}
\Psising^0(x, z_0) =
 \frac{2 i (-1)^{q-1}}{2^{2(q-1)}} 
\sum_{\underline{\alpha}, \underline{\beta}} 
\left[ \sum_{\ell=0}^{q-1}
P^{2q-2}_{\underline{\alpha}, \underline{\beta}; 2\ell}\bigl( x\bigr) 
\left( 2\pi R(x,z_0)\right)^{-\ell} \Gamma\left(\ell, 2\pi R(x,z_0) 
\right)\right] \\ \otimes 
\Omega_{q-1}(\underline{\alpha};\underline{\beta}).
\end{multline*}
We conclude that $R(x,z)^{q-1} \Psising^0(x,z)$ extends to a smooth differential $(q-1,q-1)$-form on $\Dom$.
\end{lemma}

While it should be emphasized that $\Psising$ is not a Schwartz function on 
$V$, we nonetheless define (as if $\Psising$ had weight $p+q$)
\[
\Psising(x, \tau, z) = \Psising^0(\sqrt{v}x,z) e^{\pi i 
\hlf{x}{x}\tau}\quad(\tau\in \mathbb{H}).
\] 
This is motivated by the second statement in the Proposition below.
Note 
\begin{equation}\label{eq:Psitau_int}
\Psising(x,\tau,z) = 
- \left(\int_{v}^\infty \psikm^0(\sqrt{t}x, z) \frac{dt}{t}\right)
e^{\pi i \hlf{x}{x}\tau}.
\end{equation}
From the definition of $\Psising$ and the properties of $\psikm$, we get 
\begin{proposition}\label{prop:propPsi}
Outside the singularities, $\Psising(x,\tau,z)$ has the following properties:
\begin{enumerate}
\item  For $d$  and $d^c$ the standard exterior differentials on $\mathcal{A}^{\bullet}(\Dom)$, we have outside $\Dom(x)$
\[
dd^c\,\Psising(x,\tau, z) = \phikm(x,\tau, z).
\]
\item We have 
\[
L_{p+q}\Psising(x,\tau,z) = \psikm(x,\tau,z),
\]
with the Maass lowering operator $L_{p+q}$ as before. 
\end{enumerate} 
\end{proposition}
\begin{proof}
\begin{enumerate}
\item This follows from Theorem \ref{prop:props_psikm} and the 
rapid decay of the Schwartz form $\phikm$:
\[
\begin{aligned}
dd^c\Psising(x,\tau,z) & = -\left(\int_{v}^\infty dd^c \psikm(\sqrt{t}x, z) 
\frac{dt}{t} \right)
\:
e^{\pi i\hlf{x}{x}\tau} \\
& =  -\left(\int_{v}^\infty \frac{\partial}{\partial t}\phikm^0(\sqrt{t}x, z)\, dt \right)
\:e^{\pi 
i\hlf{x}{x}\tau}
= \phikm(x,\tau, z).
\end{aligned}
\]
\item Immediately from the definition, 
\[
\begin{multlined}
L_{p+q}\Psising(x,\tau,z) = 2iv\frac{\partial}{\partial \bar\tau} 
\left(\int_{v}^\infty \psikm^0(\sqrt{t}x, z) \frac{dt}{t}\right)
e^{\pi i \hlf{x}{x}\tau} \\
= -v\left(\frac{\partial}{\partial v} \int_{v}^\infty \psikm^0(\sqrt{t}x, z) 
\frac{dt}{t}\right)e^{\pi i \hlf{x}{x}\tau} 
= \psikm^0(\sqrt{v}x, z)e^{\pi i \hlf{x}{x}\tau} = \psikm(x,\tau,z),
\end{multlined}
\]
again by rapid decay. \qedhere
\end{enumerate}
\end{proof}

\subsection{The current equation} \label{subsec:current}

We denote by $ \mathcal{A}_c^{k}(\Dom)$ the space of compactly 
supported differential forms on $\Dom$ of degree $k$. Recall that a locally integrable degree $k$-form $\omega$ on $\Dom$ defines a current, i.e., a (continuous) linear functional on the compactly supported forms of 
complementary degree, via
\[
[\omega](\eta) \vcentcolon = \int_\Dom \eta \wedge \omega \qquad \qquad \left(\eta \in \mathcal{A}_c^{2pq- k}(\Dom)\right).
\]
Furthermore, for the exterior derivatives of a current $[\omega]$ we have
\[
dd^c[\omega](\eta) :=  [\omega](dd^c\eta).
\]

The goal of this section is to prove the following generalization of the Thom Lemma, see Theorem~\ref{KM-properties} 3.

\begin{theorem}\label{thm:currenteqPsi}
Let $x \in V$ and let $\delta_{\cyclX{x}}$ denote the delta current 
 for the special cycle $\cyclX{x}$. Then
 \[
dd^c [ {\Psising}^0(x)] + (-i)^q \delta_{\cyclX{x}} = [\phikm^0(x)]
\] 
as currents on $\Gamma_x \backslash \Dom$. In other words, we have
\[
\int_{\Gamma_x \backslash \Dom} dd^c\eta \wedge {\Psising}^0(x) +
(-i)^q\int_{\cyclX{x}} \eta = \int_{\Gamma_x \backslash \Dom} \eta\wedge\phikm^0(x)
\]
for any $\eta \in \mathcal{A}_c^{2(p-1)q}(\Gamma_x \backslash \Dom)$.
\end{theorem}

We prove the theorem in the next two subsections following the same method employed in \cite{BrF04} and \cite{FK17}.

With this we can now define a Green current for the special 
cycles $\cyclX{m,h}\subset X$. Namely, for $m \in \Q$, $h\in L^\dual/L$ satisfying $m \equiv  \hlf{h}{h} 
\mod{\Z}$ and a real parameter $w>0$ we introduce the Green form of Kudla type on $X$
by setting 
\begin{equation}\label{eq:def_GKudla}
\Xi^K(m,w,h)(z) \vcentcolon =  \sum_{\substack{\lambda \in L + h \\ 
\hlf{\lambda}{\la} = m \\ \lambda\neq 0}}\Psising^0(\sqrt{2w}\lambda, z).
\end{equation}
Then by Theorem \ref{thm:currenteqPsi} we immediately obtain 
\begin{corollary}
The singular differential $(q-1,q-1)$-form $\Xi^K(m, 
w, h)$ defines a Green current for the cycle $\cyclX{m, h}$ on $X$. 
\end{corollary}

\subsubsection{Local integrability}

\begin{proposition}\label{prop:localint}
Let $x \in V$. Then $\Psising^0(x)$ and $d^c \Psising^0(x)$ are 
locally integrable differential forms on $\Dom$.
\end{proposition}

\begin{proof}  

We  view a top-degree differential form $\phi \in \mathcal{A}^{2pq}(\Dom)$ via the Hodge $\ast$-operator as a ($K$-invariant) function on $G$. We pick suitable coordinates on $\Dom$, using the decomposition 
$G=HAK$, where $H$ is the stabilizer of the first basis vector $v_1$ of $V$, $A$ is a 
one parameter subgroup $A = \left\{ a_t = \exp(tX_{1 p+q}) ;\, t\in \R 
\right\}$. Set $A_0  = \{ a_t: t \geq 0\}$. Then, see \citep[][Sec.\ 2]{F80} or \citep[][Section 2]{OT09} for 
details, 
\begin{equation}\label{eq:intpsitilde}
\begin{multlined}
\int_{\Dom} \phi= \int_{G} \phi(g)\, dg 
= C\int_{A_0}\int_H \phi(h a_t) \sinh(t)^{2q-1} \cosh(t)^{2p-1}\, dh\, dt,
\end{multlined}
\end{equation}
with $C$ a positive constant, depending on the normalization of the 
invariant measures.

Now $\Psising^0(x)$ is smooth unless $(x,x)>0$. In that case we may assume that $x = \sqrt{m} v_1$, for some $m>0$. Then for $\eta \in \mathcal{A}_c^{2(pq-(q-1))}(\Dom)$. We set $\phi = \eta \wedge \Psising(x)$ and see
\begin{gather*}
  \phi(ha_t) = \eta(h a_t) \wedge \Psising^0(a_t^{-1} h^{-1} \sqrt{m}v_1 ), \\
  \intertext{wherein}
a_t^{-1} h^{-1} \sqrt{m} v_1  = \cosh(t)\sqrt{m} v_1 - \sinh(t) \sqrt{m} v_{p+q}.
\end{gather*}
Hence,
\begin{equation}\label{eq:intliecomp}
\left( a_t^{-1} h^{-1} \sqrt{m} v_1\right)_{z_0} 
= -\sinh(t) \sqrt{m} v_{p+q} \quad \text{and}\quad  
\left( a_t^{-1} h^{-1} \sqrt{m} v_1\right)_{z_0^\perp} 
= \cosh(t) \sqrt{m} v_{1}.
\end{equation}
Thus, we have (see Lemma  \ref{lemma:psitildeexpl}),
\[
\begin{aligned}
\Psising^0( a_t^{-1} h^{-1} \sqrt{m} v_1)  =    \frac{2 i (-1)^{q-1}}{2^{2(q-1)}} &
\biggl[ \sum_{\ell=0}^{q-1}
\left( 2\pi m \sinh^2(t) \right)^{-\ell} 
\Gamma\left(\ell, 2\pi m \sinh^2(t)  \right)\\
&\quad \cdot  \;\sum_{\underline{\alpha}, \underline{\beta}}   P^{2q-2}_{\underline{\alpha}, \underline{\beta}; 2\ell}\bigl(\kappa \sqrt{m} \cosh(t) v_1 \bigr) \biggr]  \otimes \Omega_{q-1}(\underline{\alpha};\underline{\beta}).
\end{aligned}
\] 
We conclude that the integrand of \eqref{eq:intpsitilde}, i.e.,
\[
\eta(h a_t)\wedge\Psising^0(a_t^{-1} h^{-1} \sqrt{m}v_1 )\sinh(t)^{2q-1} \cosh(t)^{2p-1}
\]
is bounded, in fact, vanishes, as $t \rightarrow 0$. Further, as $\eta$ has compact support, the 
integral is convergent.


For the local integrability of $d^c\Psising(x)$ the reasoning is similar, but a bit more tedious. Again, we may assume that $x = \sqrt{m} v_1$, with $m>0$.   Further, note that we only need to consider highest-degree 
terms.

Note $d^c \Psising^0(x) = - \int_{1}^\infty 
d^c\psikm^0(\sqrt{s} x)\frac{d s}{s}$, which can be evaluated similarly to Lemma \ref{lemma:psitildeexpl}. By 
\eqref{eq:psidc_schroe},  
$d^c\psikm$ consists of two parts. Both involve polynomials of degree $2q-1$ 
which depend  on the positive coordinates of $x$ (note that there is no 
constant part). If by \eqref{eq:intliecomp}, we set $x= \cosh(t)\sqrt{m}v_1$, 
only the polynomials which depend exclusively on the first vector can 
contribute to $d^c\Psising^0(x)$. From their highest-degree terms, we get 
\[
2^{-(2q-1)} \cosh(t)^{2q-1}m^{q - \frac12} \sqrt{s}^{2q-1}.
\]
Also, in  \eqref{eq:psidc_schroe} there are linear homogeneous polynomials in 
the negative coordinates, $Q_{\alpha_q',\underline{\alpha}_{(q-1)}}$ and 
$Q_{\alpha_q,\underline{\alpha}_{(q-1)}'}'$. From them, again by 
\eqref{eq:intliecomp} we have contributions of
\[
\begin{gathered}
- \sqrt{s}\ \sqrt{m}  \sinh(t) 
\end{gathered}
\]
Hence, gathering the contributions of the non-vanishing highest-degree terms, 
we still have the integral
 \[
\begin{multlined}
\int_{1}^\infty 
 s^{q-1} e^{2\pi R(x,z_0)} ds   
 = (2\pi R(x,z_0))^{-q}\, \Gamma(q, 2\pi R(x,z_0)) \\
 = (2\pi\sinh^2(t))^{-q}\, \Gamma(q, 2\pi \sinh^2(t)).
\end{multlined}
\]
Thus, up to sign, for $t\rightarrow 0$ the 
behaviour of  $d^c\Psising( a_t^{-1} h^{-1} 
\sqrt{m}v_1)$ is dominated by terms of the form
\begin{equation}\label{eq:dcPsi_v1}
\frac{(-1)^{q-1}\pi}{2^{2q-1}} 
\sinh(t) \cosh(t)^{2q-1} \left(\sinh^2(t)\right)^{-q}\Gamma\left( q, 2\pi 
\sinh^2(t) \right).
\end{equation} 
In particular, it follows that the integrand in 
\[
\int_{A_0}\int_H \eta(h a_t) \wedge \bigl(d^c\Psising(a_t^{-1} h^{-1} \sqrt{m}v_1 )\bigr)  \sinh(t)^{2q-1} \cosh(t)^{2p-1}\, dh\, dt,
\]
remains bounded as $t\rightarrow 0$, and hence the integral converges. 
\end{proof}

\subsubsection{The current equation}\label{par:current_eqPsi}

\begin{proof}[Proof of Theorem~\ref{thm:currenteqPsi}]

Let $\eta \in \mathcal{A}_c^{2(p-1)q}(\Gamma_x \backslash \Dom)$, not necessarly closed.
First note using $(dd^c \eta)\wedge \Psising^0(x)  = (d \eta) \wedge d^c \Psising^0(x)
-d^c\left(d\eta \wedge \Psising^0(x)\right)$ and Stokes' theorem
\[
 \int_{\Gamma_x \backslash \Dom} (dd^c \eta)\wedge \Psising^0(x)
 = - \int_{\Gamma_x \backslash \Dom}  (d \eta) \wedge d^c \Psising^0(x) + 
   \lim_{\epsilon \rightarrow 0}\int_{\Gamma_x\backslash\partial\left( \Dom - U_\epsilon(x) \right) }(d \eta) \wedge d^c \Psising^0(x),
 \]
where $U_\epsilon$, $(\epsilon > 0)$ denotes an open neighbourhood of the 
cycle $\cycl{x}$.
Next we show that the limit on the right hand side 
vanishes. We may again assume 
$x=\sqrt{m}v_1$, with $m>0$ and use the $HAK$ coordinates introduced in the proof of Proposition~\ref{prop:localint}. Then for $\epsilon>0$, an open neighborhood of 
$\cycl{v_1}$ is defined by
\begin{equation}\label{eq:defUeps}
U_\epsilon = \Dom - \left(  H \times A_\epsilon \right),
\end{equation}
with $A_\epsilon = \{ a_t: t\geq \epsilon  \}$.
With the analog of the integral formula from \eqref{eq:intpsitilde}, the limit  
can be written as
\[
C\lim_{\epsilon\rightarrow 0} 
\int_{\Gamma_{v_1}\backslash H}\eta(h a_\epsilon)\wedge\Psising^0( a_\epsilon^{-1} h^{-1} \sqrt{m}v_1 )\sinh(\epsilon)^{2q-1} \cosh(\epsilon)^{2p-1}\, dh
\]
for some constant $C$.
Only the highest degree term  of $\Psising( a_t^{-1} h^{-1} 
\sqrt{m}v_1)$ (see Lemma  \ref{lemma:psitildeexpl})  can contribute. 
Further, note that, since  $\left( a_t^{-1} h^{-1} \sqrt{m} 
v_1\right)_{z_0^\perp} = \cosh(t) \sqrt{m} v_{1}$ by \eqref{eq:intliecomp}, we 
have 
$P^{2q-2}_{\underline{\alpha}, \underline{\beta}; 2q-2}\left(
\sqrt{{m}} \cosh(t) v_1 \right)\neq 0$ only for $\underline{\alpha} = 
\underline{\beta}= (1, \dotsc, 1)$, thus, up to constants, the highest degree term is given by
\[
\left( m \sinh^2(t) \right)^{-(q-1)} 
\Gamma\left(q-1, 2\pi m \sinh^2(t)  \right)
(2 \sqrt{m}\cosh(t))^{2q-1}.
\] 
Hence, comparing powers of $\sinh(t)$ we see that the integrand goes to zero 
for $t = \epsilon \rightarrow 0$, and the limit vanishes as claimed. 

Now, since $dd^c \Psising(x) = \phikm(x)$, we have
\[
  \begin{aligned}
  - \int_{\Gamma_x \backslash \Dom} (d \eta) \wedge d^c \Psising^0(x)
& =\int_{\Gamma_x \backslash \Dom}\eta \wedge dd^c \Psising^0(x) 
-\int_{\Gamma_x \backslash \Dom} d\left( \eta\wedge d^c{\Psising}^0(x) \right) \\
& = \int_{\Gamma_x \backslash \Dom}\eta \wedge \phikm^0(x)  
+ \lim_{\epsilon\rightarrow 0} \int_{\Gamma_x\backslash\partial\left( \Dom - U_\epsilon(x) \right)} \eta \wedge d^c{\Psising}^0(x), 
\end{aligned}
\]
again by applying Stokes' theorem.
Thus it remains to show that 
\[
   \lim_{\epsilon\rightarrow 0} \int_{\Gamma_x\backslash\partial\left( \Dom - U_\epsilon(x) \right)} \eta \wedge d^c{\Psising}^0 (x)
=(-i)^q \int_{\cyclX{x}}\eta.
\]
We have to consider the limit of the same integral as in the proof of second 
part of Proposition \ref{prop:localint}: 
\begin{equation}\label{eq:limPsibar}
C \,\lim_{\epsilon \rightarrow 0} \int_{\Gamma_{v_1}\backslash H} \eta(h a_\epsilon) \wedge  d^c\Psising^0(a_{-\epsilon}h^{-1} \sqrt{m} v_1) \cosh(\epsilon)^{2p-1}\sinh(\epsilon)^{2q-1}\, dh,
\end{equation}
with a non-zero constant $C$, independent of $\eta$. 
With \eqref{eq:dcPsi_v1} we see that for both parts of $d^c\Psising^0(x)$, the 
integral is bounded as $t = \epsilon \rightarrow 0$. We have
\[
\begin{multlined}
C \,\lim_{\epsilon \rightarrow 0} \int_{H} \eta(h a_\epsilon) \wedge \bigl( d^c\Psising^0(a_{-\epsilon}h^{-1} \sqrt{m} v_1) \bigr) \cosh(\epsilon)^{2p-1}\sinh(\epsilon)^{2q-1}\, dh\\
=  \tilde{C} \int_{H}\eta(h)\, d h,
\end{multlined}
\]
with a constant $\tilde{C}$ independent of $\eta$. By Kudla-Millson 
theory \citep[see][Theorem 6.4]{KM87}, we see that $\tilde{C} = 
(-i)^{q}$ for $\eta$ closed, see Theorem \ref{KM-properties} 3.

To summarize, we have showed that for all 
$\eta\in\mathcal{A}_c^{2(p-1)q}(\Gamma_x \backslash \Dom)$, 
\[
  \int_{\Gamma_x \backslash \Dom} (dd^c \eta)\wedge \Psising^0(x)
 = \int_{\Gamma_x \backslash \Dom}\eta \wedge \phikm^0(x) - (-i)^q\int_{\cyclX{x}}\eta,
\]
as claimed.
\end{proof}

\section{The singular theta lift}

\subsection{Weak Maass forms}
We let $L$ be an even Hermitian lattice with the Hermitian form $\hlfempty$. Further, we denote by $L^-$ the same $\mathcal{O}_{\F}$ module $L$ but with the Hermitian form  $-\hlfempty$.

We denote the standard basis elements of $\C[L^\dual/L]$ by $\ebase_h$ $(h\in L^\dual/L)$ 
 and introduce the Hermitian pairing
\[
\left\langle\, ,\, \right\rangle_L : \C[L^\dual/L] \times \C[L^\dual/L] \rightarrow \C, \quad\text{by setting}
\left\langle \ebase_\mu , \ebase_\nu \right\rangle_L
= \delta_{\mu,\nu}\quad(\mu,\nu \in L^\dual/L).
\]
Similar definitions are made for the lattice $L^-$.

Recall there is a finite Weil representation of $\SL_2(\Z)$ on $\C[L^\dual/L]$, which we denote by $\omega_L$. It is most easily described through the action of the generators $S = \begin{psmallmatrix} 0 & -1 \\ 1 & \phantom{-}0
\end{psmallmatrix}$ and $T = \begin{psmallmatrix}
1 & 1 \\ 0 & 1
\end{psmallmatrix}$:
\[
\begin{gathered}
\omega_L(T) \ebase_h = e(\hlf{h}{h})\ebase_h, \qquad
\omega_L(S) \ebase_h = \frac{i^{q-p}}{\sqrt{\abs{L^\dual/L}}}\sum_{\mu \in L^\dual/L} e(-2\hlf{\mu}{h}_\R)\ebase_\mu.
\end{gathered}
\]
We denote by $\omega_L^\vee \simeq  \bar\omega_L$ the dual representation. Note that $\omega_L^\vee = \omega_{L^-}$.
For $k \in \Z$ and $\gamma \in \SL_2(\Z)$, define the weight 
$k$-slash operation on functions $\C[L^\dual/L] \rightarrow \C$ 
as
\[
f\mid_{k,L} \gamma = (c\tau + d)^{-k}\omega_L(\gamma)^{-1} f(\gamma \tau). 
\] 
The slash-operation for  the dual representation is defined similarly.  

Following \cite{BrF04} and \cite{ES16} we now define several spaces of modular forms. 
\begin{definition}[\protect{\citep[see][Section 3]{BrF04}}] For $k\in\Z$, let $\HmfL{k}{L}$ be the space of twice continuously differentiable functions $f: \mathbb{H}\rightarrow \C[L^\dual/L]$, which satisfy
\begin{enumerate}
\item $f\mid_{k,L}(\gamma) = f$ for all $\gamma \in \SL_2(\Z)$.
\item There exists a constant $C>0$ such that $f(\tau) = {O}(e^{Cv})$ as $v\rightarrow \infty$. 
\item $\Delta_k f = 0$.  
\end{enumerate} 
\end{definition}
The elements of $\HmfL{k}{L}$ are called harmonic weak Maass forms. 
Any such form $f$ has a decomposition $f(\tau) = f^+(\tau) + f^-(\tau)$ into a holomorphic and a non-holomorphic part, where the Fourier expansion of the holomorphic part is
\[
f^+(\tau) = \sum_{h\in L^\dual/L} \sum_{n\in \Q} a^+(h,n) e(n\tau)\,\ebase_h,
\]
whilst that of the non-holomorphic part is 
\[
f^-(\tau) = \sum_{h\in L^\dual/L}\Bigl(
a^-(h,0)v^{1-k} 
+ \sum_{\substack{n\in \Q \\ n\neq 0} }a^-(h,n)
\Gamma(1-k, 4\pi n v)e(nu) \Bigr)\,\ebase_h.
\]
We denote by $P(f)$ the principal part of $f$, i.e.\ the Fourier polynomial 
\[
P(f)(\tau) = P(f^+)(\tau) = \sum_{h\in L^\dual/L} \sum_{\substack{n\in\Q \\ n < 0}}
a^+(h,n) e(n\tau)\ebase_h. 
\]
Note that $\HmfL{k}{L}$ contains the spaces of weakly holomorphic modular forms $\MfLw{k}{L}$ and holomorphic modular forms $\MfL{k}{L}$, with $ \HmfL{k}{L}\supset \MfLw{k}{L} \supset \MfL{k}{L}$. 

The operator $\xi_k$ affords an antilinear mapping given by 
\[
\xi_k: \HmfL{k}{L}\longrightarrow \MfLw{2-k}{L^-}, \qquad
f(\tau)\longmapsto 2iv^{k} \overline{\frac{\partial f(\tau) }{\partial \bar{\tau}}} = 
v^{k-2}\overline{L_k f(\tau)}.
\]
Now, the space $\HmfpL{k}{L}$ is defined as the inverse image of the cusp forms $\CfL{2-k}{L^-}$.
It follows immediately from this definition that for $f\in \HmfpL{k}{L}$,
\[
f(\tau) - P(f)(\tau) = 
\mathbf{O}(e^{-Cv}),
\] 
as $v\rightarrow \infty$ for some constant $C>0$. 

Further, by \citep[][Corollary 3.8]{BrF04}, there are exact sequences 
\begin{gather*}
0 \longrightarrow \MfLw{k}{L} \longrightarrow \HmfL{k}{L} \stackrel{\xi_k}{\longrightarrow}  \MfLw{2-k}{L^-} \longrightarrow 0, \\ \shortintertext{and}
0 \longrightarrow \MfLw{k}{L} \longrightarrow \HmfpL{k}{L} \stackrel{\xi_k}{\longrightarrow}  \CfL{2-k}{L^-} \longrightarrow 0. 
\end{gather*}
Following Ehlen and Sankaran \cite{ES16}, we generalize this setup by introducing two further spaces of modular forms, $\AmodL{k}{L^-}$ and $\AweakL{k}{L^-}$. For the former space, we use the following,  slightly modified definition from \citep[][Definition 3.2]{BES18}: 
\begin{definition} \label{def:Amod}
Let $\Amod{k}(\omega_L^\vee) = \AmodL{k}{L^-}$ denote the space of $\mathcal{C}^\infty$-functions $f: \mathbb{H}\rightarrow \C[L^\dual/L]$ satisfying
\begin{enumerate}
\item $f\mid_{k,L^-}(\gamma) = f$ for all $\gamma \in \SL_2(\Z)$.
\item For all $a, b \in \Z_{\geq 0}$, there is an $r\in\Z$ such that
  $\frac{\partial^a}{\partial^a u}\frac{\partial^b}{\partial^b v} f(\tau) =  \mathbf{O}(v^r)$ as $v\rightarrow\infty$.
\item If $f = \sum_{m\in\Q} c(m,v)e( m \tau)$ denotes the Fourier expansion of $f$, then the integral
  \[
\int_{1}^ \infty c(0,t) t^{-2-s} dv,
\]
has a meromorphic continuation to a half-plane $\Re(s)>-\epsilon$ for some $\epsilon >0$. (The integral converges for sufficiently large $\Re(s)\gg 0$, since by 2., $f$ is of polynomial growth as $v\rightarrow \infty$.)  
\end{enumerate} 
\end{definition}
\begin{definition}[\protect{\citep[see][Definition 2.8]{ES16}}]\label{def:Abang}
Denote by $\Aweak{k}(\omega_L^\vee)=\AweakL{k}{L^-}$ the space of $\mathcal{C}^\infty$ functions $f: \mathbb{H}\rightarrow \C[L^\dual/L]^\vee$ satisfying
\begin{enumerate}
\item $f\mid_{k,L^-}(\gamma) = f$ for all $\gamma \in \SL_2(\Z)$.
\item There exists a constant $C>0$ such that $f(\tau) = \mathbf{O}(e^{Cv})$ as $v\rightarrow \infty$. 
\item  $L_k(f) \in \Amod{k-2}(\omega_L^\vee)$.
\end{enumerate}
\end{definition} 

\subsection{Regularized theta integral}

In the following, we set $\kappa = p+q-2$ and $k =-\kappa= -(p+q)+2$.

For $h \in L^\dual/L$ we want to define for the Schwartz form $\psikm$ introduced in Section~\ref{sec:Millson-form} the theta function component $\theta(\tau, z, \psikm)_h$ to obtain the vector-valued theta function 
\[
\Theta(\tau, z) := \Theta(\tau, z, \psikm) \vcentcolon = \left( \theta(\tau, z, \psikm)_h \right)_{h \in L^\dual/L} =
\sum_{h \in L^\dual/L} \theta(\tau, z, \psi)_h \ebase_h.
\]
There is a technicality. The form $\phi(\lambda,\tau)= v\psi^0(\sqrt{v}\lambda)e^{\pi i 
\hlf{\lambda}{\lambda}\tau}$, see \eqref{phixtau}, does not give rise for $h=0$ (say) to a proper $q$-series upon summation over $L$ since $(\lambda,\lambda) \in \Z$ but not necessarily even. We therefore set (and analogously for $\varphi_{KM}$)
\[
\theta(\tau, z, \psikm)_h = \sum_{\lambda \in L + h}
\psikm(\sqrt{2}\lambda, \tau, z), \qquad (\tau\in\h, z\in \Dom)
\]
and then the theta function $\Theta(\tau, z)$ does transform like a vector-valued modular form of type $\rho_L$ for $\mathrm{SL}_2(\Z)$ with weight $\kappa$. We can view this procedure as first switching to the Weil representation associated to the additive character $t \mapsto e(2t)$ before applying the summation. Explicitly, 
\begin{equation}\label{eq:FourierTheta_expl}
\Theta(\tau, z) =
v\sum_{h \in L^\dual/L} \sum_{\lambda \in L + h} 
P_\psi(\sqrt{2v}\lambda,z) e^{4\pi v\hlf{\lambda_z}{\lambda_z} + 2\pi i \tau \hlf{\lambda}{\lambda}}\ebase_h,
\end{equation}
where $P_\psi(x, z) \in \bigl[\mathcal{P}(V)\otimes \mathcal{A}^{\bullet}(\Dom)\bigr]^G$ is the polynomial part of $\psi$, see \eqref{psi-poly}. 

Following \cite{Bo98,Br02,BrF04}, for a weak harmonic Maass form $f \in \HmfpL{k}{L^-}$, we 
consider the regularized theta integral 
\[
\int_{\mathrm{SL}_2(\Z) \backslash \h}^{reg} \left\langle f(\tau), 
\overline{\Theta(\tau, z)} \right\rangle_L  d\mu.
\]
Due to the invariance under $\SL_2(\Z)$, the following regularization recipe 
(due to Harvey and Moore \cite{HM96}) can be used: For $t\in 
\R_{>0}$, denote by 
$\mathcal{F}_t$ the truncated fundamental domain given by
\[
\mathcal{F}_t\vcentcolon 
= \left\{ \tau = u + iv; \abs{\tau} > 1, {\textstyle -\frac12 < u < \frac12}, 0 < v \leq t \right\},
\]
and define
\begin{equation}\label{eq:defthetareg1}
\Phi(z, f, \psikm) \vcentcolon =
\int_{ \mathcal{F} }^{reg} \left\langle f, \overline{\Theta(\tau, z)} \right\rangle_L  d\mu \vcentcolon=
 \CT_{s=0}\left[ \lim_{t\rightarrow \infty} \int_{ \mathcal{F}_t}  \left\langle f, \overline{\Theta(\tau, z)} \right\rangle_L v^{-s}\, d\mu \right],
\end{equation}
where the notation $ \CT_{s=0}$ denotes the constant term at $s=0$ of the meromorphic continuation of the limit. 

\label{p:reg_pairing}
More generally, we introduce a regularized pairing as follows. For $f \in 
\HmfpL{k}{L^-}$ and $g$ transforming as a modular form of weight $\kappa$ under 
$\omega_{L}$ set 
\begin{equation}\label{eq:def_pairreg}
\pairregLm{f}{g} =  \CT_{s=0}\left[ \lim_{t\rightarrow \infty} \int_{ 
\mathcal{F}_t}  \left\langle f, \bar{g} \right\rangle_L v^{-s}\, d\mu \right].
\end{equation}
We say that \emph{the pairing exists} if for sufficiently large $\Re(s)$ the 
limit $t\rightarrow\infty$ defines a holomorphic function in $s$ for which a 
meromorphic continuation to some $\Re(s) <0$ exists, so that constant of the 
Laurent expansion around $s=0$ can be evaluated.

\subsection{Singularities and current equation}\label{subsec:sing_of_Phi}
Let $f$ be a harmonic weak Maass form with holomorphic Fourier coefficients $a^+(h,n)$, 
$h\in L^\dual/L$, $n\in \Q_{<0}$. We define a locally finite cycle $\cycl{f} $ on $\Dom $ by 
\[
\cycl{f} \vcentcolon = \sum_{h \in L^\dual/L}\sum_{n\in \Q_{<0}}a^+(h,n) \cycl{n,h}
\] 
and denote by $\cyclX{f}$ the image of $\cycl{f}$ on $X$. 
\begin{proposition}\label{prop:singularities}
The regularized lift $\Phi(z, f, \psikm)$ converges to a smooth differential form on $\Dom$ with singularities along the cycle $\cycl{f}$. In a small neighbourhood of $w\in\Dom$, the singularities are of type
\[
-\sum_{h\in L^\dual/L}\sum_{\substack{n \in \Q \\ n<0}} 
a^+(h, n) 
\sum_{\substack{\lambda \in L + h  \\ \hlf{\la}{\la} = -n \\ \lambda \in w^\perp }}
\Psising^0(\sqrt{2}\lambda, \tau, z), 
\]
i.e., the difference of $\Phi(z, f, \psikm)$ and this sum extends to a smooth form.
\end{proposition}
\begin{proof}
The argument closely follows \citep[][Sec.\ 5]{BrF04}.
It suffices to consider the integral up to smooth functions. Due to the rapid decay of the non-holomorphic part of $f$, the integral converges for $f^-$ to a smooth form, and we only need to consider 
\[
\sum_h \lim_{t\rightarrow \infty} \int_{\mathcal{F}_t}^{reg} f^+_h(\tau) \theta(\tau, z, \psikm)_h v^{-s}\, d \mu. 
\]
Also, since the integral over $\mathcal{F}_1$ is smooth, it suffices to consider the integral over $v>1$:
\begin{equation}\label{sing-integral}
\sum_h \lim_{t\rightarrow \infty} \int_{1}^{t}\int_{- \frac12}^{\frac12}   f^+_h(\tau) \theta(\tau, z, \psikm)_h v^{-s-2} du\, dv. 
\end{equation}
Now, the integration over $u$ picks out the constant term in the Fourier expansion of the integrand, which in the notation of 
 \eqref{eq:FourierTheta_expl} is given by
\[
v\sum_h\sum_{\lambda \in L + h} a^+(h, - \hlf{\la}{\la}) P_\psi(\sqrt{2v}\lambda, z) e^{4\pi v\hlf{\lambda_z}{\lambda_z}}.
\]
For \eqref{sing-integral} we therefore obtain 
\begin{equation}\label{eq:sum_la1}
\sum_{\la \in L^\dual} a^+(\la, -(\la,\la))\int_{1}^\infty P_\psi(\sqrt{2v}\lambda, z) e^{4\pi v\hlf{\lambda_z}{\lambda_z}} v^{-s-1} dv.
\end{equation}
For a relatively compact open neighbourhood $U \subset \Dom$, define the set
\[\label{def:sfU}
S_f(U,\epsilon) = \left\{ \lambda \in L^\dual\, ;\, a^+\bigl(\lambda, -\hlf{\lambda}{\lambda}\bigr)\neq 0 \quad\text{and}\:\; \abs{\hlf{\la_z}{\la_z}}<\epsilon \quad\text{for some}\; z\in U  \right\}. 
\]
By reduction theory, this set is finite, as $f^+$ has only finitely many non-vanishing Fourier coefficients in its principal part. 

Using standard arguments, like in \cite{BrF04}, one finds that in \eqref{eq:sum_la1} the sum of all terms with $\lambda \in L^\dual - S_f(U,\epsilon)$ is majorized by
a convergent sum, $\sum_{\la \in L^\dual} \exp\left( -C \hlf{\la}{\la}_z\right)$ for some $C>0$, and hence converges. 
Further, in \eqref{eq:sum_la1}, the term with $\lambda = 0$ is given by $
a^+(0,0) P_\psi(0, z)\int_{1}^\infty \frac{1}{v^{s+1}} dv$, which falls out after regularization.

Finally, all that remains of \eqref{eq:sum_la1} is the following finite sum, which dictates the singularities in $U$:  
\[
\sum_{0\neq \la \in S_f(U, \epsilon)} a^+(\la, -(\la,\la))\int_{1}^\infty P_\psi(\sqrt{2v}\la, z) e^{4\pi v\hlf{\lambda_z}{\lambda_z}} v^{-s-1} dv.
\]
Clearly, the integral has meromorphic continuation to the entire $s$-plane, and for $s=0$ is equal to $-\Psising^0(\sqrt{2}\lambda, \tau, z)$, cf. \eqref{def:Kudla-xi}. Hence, the singularity for $z\in U$ is dictated by 
\[
-\sum_{\substack{\lambda \in S_f(U, \epsilon)\\ \lambda\neq 0}} a^+(\lambda, - (\la,\la)) \Psi^0(\sqrt{2}\lambda, \tau, z).
\]
In particular, $z$ is a singular point precisely if  $R(\lambda, z) = - \hlf{\lambda_z}{\lambda_z} = 0 $  for some $\lambda \in S_f(U, \epsilon) - \{0\}$. 
\end{proof}

\paragraph{The singular theta lift as a current}\label{par:liftascurrent}
Using the relationship between the singular theta lift  and the singular 
Schwartz form $\Psising$, already seen in the proof of Proposition 
\ref{prop:singularities}, we derive a current equation for $\Phi(f,\psi)$. The role of $\phikm$ in Theorem \ref{thm:currenteqPsi} is now played by 
\begin{equation}\label{eq:defLambda_psi}
\Lambda_{\psikm}(f) \vcentcolon= dd^c \Phi(z, f, \psikm),
\end{equation}
where $f \in \HmfpL{k}{L^-}$.

\begin{theorem}\label{prop:liftascurrent}
The singular theta lift $\Phi(z, f,\psi)$ and the lifting $\Lambda_\psi(f)$ 
satisfy the following current equation on $X$:
\[
dd^c [\Phi(f,\psi)] + (-i)^q\delta_{\cyclX{f}} = [\Lambda_\psi(f)].
\]
\end{theorem}

\begin{proof}
This follows directly from Theorem \ref{thm:currenteqPsi}. For $x \in V$, we have
\begin{equation}\label{eq:currenteqPsibis}
dd^c [ {\Psising^0}(x)] + (-i)^q\delta_{\Gamma(x) \backslash\Dom(x)} = [\phikm^0(x)].
\end{equation}
As usual, denote the Fourier coefficients of $f^+$ by $a^+(\lambda, n)$ for 
$\lambda \in L^\dual$, $n\in \Q$. For any relatively compact open neighbourhood 
$U\subset \Dom$ and any $\epsilon >0$, we consider the set $S_f(U,\epsilon)$ 
from p.\ \pageref{def:sfU}. Then, from the left hand side of 
\eqref{eq:currenteqPsibis}, we get 
\[
dd^c\sum_{\substack{\lambda\in S_f(U,\epsilon) \\ \lambda\neq 0}}a^+(\lambda, 
-(\la,\la))\left[ \Psising^0(\sqrt{2}\lambda) \right] +(-i)^q \sum_{\substack{\lambda\in 
S_f(U,\epsilon) \\ \lambda\neq 0}}a^+(\lambda, -(\la,\la)) 
\delta_{\cyclX{\lambda}}.
\]
Now, by Proposition \ref{prop:singularities}, and after taking the (locally 
finite) union over neighbourhoods $U$ containing singular points, we get the 
current associated to (the singular part of) $\Phi(z,f,\psi)$ plus the delta 
current for the cycle $\cyclX{f}$:
\[
dd^c [\Phi(f,\psi)] + (-i)^q\delta_{\cyclX{f}}.
\]
(Note that, through Stokes' theorem, the current is determined by the singular 
part.)

Repeating the same steps on the right hand side of \eqref{eq:currenteqPsibis}, 
by using the identity $dd^c \Psising(x,\tau,z) = \phikm(x,\tau, z)$ (see 
Proposition \ref{prop:propPsi}), we recover the current 
\[ [dd^c \Phi(f,\psi)] = [\Lambda_\psi(f)],\]
as claimed.
\end{proof}

\subsection{Adjointness to the Kudla-Millson lift}

We now show an adjointness result analogous to \citep[][Theorem 6.1, Theorem 6.2]{BrF04}.

Denote by $\Theta(\tau, z, \phikm)$ the theta function for the Schwartz form 
$\phikm$ from Section \ref{sec:Schwartz} (see, \cite{KM86, KM87, KM90}). By 
Proposition \ref{prop:props_psikm} it is a closed differential $(q,q)$-form (in $z$), 
which has weight $p+q$ as a modular form (in $\tau$). The Kudla-Millson lift 
$\Lambda_{KM}$ is now defined for any rapidly decreasing $2(p-1)q$-form $\eta$ 
through the assignment
\[
\eta \longmapsto \Lambda_{KM}(\eta) \vcentcolon = \int_X \eta \wedge 
\Theta(\tau, z, \phikm). 
\]
This map factors through the de Rham cohomology with compact supports on $X$. By 
\citep[][Theorem 2]{KM90} if $\eta$ is closed, $\Lambda(\tau, \eta)$ is a 
holomorphic modular form.

To facilitate notation, we introduce a pairing $\left\{\cdot,\cdot 
\right\}'$ between 
the spaces $\MfL{k}{L^-}$ and $\HmfpL{k}{L}$ see \citep[][(3.15) on p.\ 
62]{BrF04}.
Let $f \in \HmfpL{k}{L}$ with $f^+ = 
\sum_{h,n} a^+(h,n)e(n\tau)\ebase_h$ 
and $h \in \MfL{k}{L^-}$ with $q$-expansion $h = \sum_{h,n} 
b(h,n)e(n\tau)\ebase_h$. We set 
\[
 \left\{ h, f \right\}' \vcentcolon 
= \left(h, \xi_k(f)\right)_{2-k,L} - \sum_{h\in L^\dual/L} a^+(h,0)b(h,0)
= \sum_{h \in L^\dual/L}\sum_{\substack{n\in\Q \\ n < 0}} a^+(h,n) b(h, -n).
\]
\begin{theorem}\label{B-KM-duality-Th}
The lift $\Lambda_\psikm$ has the following properties:
\begin{enumerate}
\item Let $f\in\HmfpL{k}{L^-}$. Then 
\[
\left( \Theta(\cdot, z, \phikm), \xi_k(f) \right)_{2 -k, L}  + a^+(0,0)\phikm(0)
= \Lambda_{\psikm}(f)
\]
as differential forms on $X$. In particular,  $\Lambda_{\psikm}(f)$ extends to a smooth closed $(q,q)$-form of moderate growth. 
\item The Kudla-Millson lift $\Lambda_{KM}$ and $\Lambda_\psikm$ are adjoint in the sense that 
\[
\left( \eta, \Lambda_{\psikm}(f) \right)_X = 
\left\{ \Lambda_{KM}(\eta), f\right\}'
\]
for any $f\in\HmfpL{k}{L^-}$ and any rapidly decreasing closed $2(p-1)q$-form 
$\eta$.
\end{enumerate}
\end{theorem}
We note that, in particular, if $f\in \MfLw{k}{L^-}$, we have
$\Lambda_{\psikm}(f) = a^+(0,0)\phikm(0)$. 
\begin{corollary} 
For any rapidly decreasing closed $2(p-1)q$-form $\eta$ and any $f\in\HmfL{k}{L}$, we have
\[
\left(\eta, \Lambda_{\psikm}(f)\right)_X = \int_{\cyclX{f}} \eta.
\]
\end{corollary}
\begin{proof}[Proof of the Theorem]
\begin{enumerate}
\item
We have 
\[
L_{2-k}\Theta(\tau, z,  \phikm) =  \Theta(\tau, z, dd^c\psikm),
\] 
 since $L\phikm(0) = dd^c \psikm(0)$,
Hence, we have
\[
\begin{aligned}
\lim_{t\rightarrow \infty} \int_{\mathcal{F}_t}
\left\langle L_{2-k}\Theta(\tau, z, \phikm), \bar{f} \right\rangle d\mu
=  \int_{\mathcal{F}}^{reg} \left\langle L_{2-k}\Theta(\tau, z, \phikm), \bar{f} \right\rangle d\mu \\
=  \int_{\mathcal{F}}^{reg} \left\langle \Theta(\tau, z, dd^c \psikm), \bar{f} \right\rangle d\mu,
\end{aligned}
\]
and this quantity defines a smooth form on $\Dom - \cycl{f}$, which extends smoothly to $\Dom$. With \citep[][Lemmas 6.6, 6.7]{BrF04} we get the following identity, valid outside $\cycl{f}$:
\[
\left( \Theta(z, \phikm), \xi_k(f) \right)_{2 -k, L} =  \int_{\mathcal{F}}^{reg} \left\langle \Theta(\tau, z, dd^c \psikm), \bar{f} \right\rangle d\mu + a^+(0,0) \phikm(0).
\] 
Now, the statement follows by showing that 
\begin{equation}\label{eq:switchddcint}
 \int_{\mathcal{F}}^{reg} 
\left\langle \Theta(\tau, z, dd^c \psikm), \bar{f} \right\rangle d\mu 
= dd^c   \int_{\mathcal{F}}^{reg} \left\langle \Theta(\tau, z, \psikm), \bar{f} \right\rangle d\mu.
\end{equation}
First, note that 
\begin{equation}\label{eq:regreally}
\int_{\mathcal{F}}^{reg} \left\langle \Theta(\tau, z, \psikm), \bar{f} \right\rangle d\mu =
\lim_{t\rightarrow\infty}
\int_{\mathcal{F}_t} \left( \left\langle \Theta(\tau, z, \psikm), \bar{f} \right\rangle   - a^+(0,0)v  \right) d\mu 
+ C a^+(0,0),
\end{equation}
with a constant $C$, coming from the regularisation of the constant term.
Arguing along the same lines as in the proof of Proposition \ref{prop:singularities}, 
we see that in the integrand, the sum over $\lambda \in L^\dual - S_f(U,\epsilon)$ (see p.\ \pageref{def:sfU}) converges uniformly for any relatively compact open neighbourhood $U\subset\Dom$  and any $\epsilon>0$. For the remaining terms, with $\lambda\in S_f(U, \epsilon)$ the integrand decays exponentially. 

Thus, switching the order of differentiation from the right hand side of \eqref{eq:switchddcint} and the limit from \eqref{eq:regreally} is justified, which completes the proof. 
\item The second statement follows from the first, the proof is 
exactly like the one of \citep[][Theorem 6.3]{BrF04}, which we briefly 
reproduce here. Denote by $\left(\cdot, \cdot\right)_X$ the natural pairing 
between closed forms of complementary degree (where one is rapidly decreasing 
and the other of moderate growth). We have
\[
\begin{aligned}
\left( \eta, \Lambda_{\psikm}(f) \right)_X &  = 
\bigl( \eta, \left( \Theta(\cdot, z, \phikm), \xi_{k}(f)\right)_{k,L}\bigr)_X
\\  & = \left( \left( \eta,  \Theta(\cdot, z, \phikm)\right)_{X}, 
\xi_{k}(f)\right)_{k,L}
=
\left\{ \Lambda_{KM}(\eta), f\right\}.
\end{aligned}
\]
Note only that the order of integration can be switched by absolute convergence.
\end{enumerate}
\end{proof}

\section{Comparison of the two Green forms}
In this section, we compare the Green forms of Kudla type $\Green^K(m,w,h)$, for $m\in \Q$, $h\in L^\dual/L$ and $w\in \R_{>0}$, and those of Bruinier type $\Green^B(m,h)$ (see below).
The aim is to transfer some of the results of Ehlen and Sankaran from \cite{ES16} to the present setting.

\subsection{Green form of Bruinier type}

We first introduce the Green form of Bruinier type. 

The Hejhal Poincar\'{e} series (also known as Maass-Poincar\'{e} series) of weight 
$k$ of index $(m,h)$, $h\in L^\dual/L$, $m\in \Z$ is defined as (for $\tau\in 
\mathbb{H}$, $s \in \C$ with $\sigma = \Re(s)>1)$
\begin{equation}\label{eq:defFmh}
F_{m,h}(\tau,s ) = \frac{1}{4\Gamma(2s)} 
\sum_{A\in \Gamma_\infty \backslash \SL_2(\Z)} \mathcal{M}_s(4\pi \abs{m} v) 
e^{2\pi i m u}\ebase_h\mid_{k,L^-} A,
\end{equation}
where $\mathcal{M}_s(t) = t^{-\frac{k}{2}}M_{-\frac{k}{2}, s - \frac12}(t)$, 
with the M-Whittaker function $M_{\kappa,\mu}(t)$. 
Note that our definition of $F_{m,h}(\tau,s)$ differs from \citep[][Definition 
1.8]{Br02} by a factor of $\frac12$.

Set $s_0 = 1 -  \frac{k}{2}$. For fixed $s = s_0$, the Poincar\'{e} series 
$F_{m,h}(\tau, s_0)$ have principal part $ q^m \ebase_h$ and form a  basis of $\HmfpL{k}{L^{-}}$, 
\citep[see][Proposition 1.12]{Br02}. Note further that by \citep[][Remark 
3.10]{BrF04}  $\xi_{k}(F_{m,h}(\tau, s_0))$ is a 
holomorphic, cuspidal Poincar\'{e} series of index $(-m, h)$. 

We now introduce two Green forms hrough the regularised pairing (see p.\  
\pageref{eq:def_pairreg}) of the Hejhal Poincar\'{e} series with $\Theta(\tau, 
z)$. 
First, we define the Bruinier type Green form
$\Green^B(m,h)$ by setting
\begin{equation}\label{B-current}
\Green^B(m,h)(z) \vcentcolon =  \pairregLm{F_{m,h}(\tau,s_0)}{\Theta(\cdot, 
z)},
\end{equation}
i.e., the regularised theta lift of the weak Maass form $F_{m,h}(\tau,s_0)$.
By Theorem~\ref{prop:liftascurrent} 
$\Green^B (m,h)$ is thus a Green current for the cycle $\cyclX{m,h}$.

\subsection{The Kudla type Green form as a theta lift}
Following \citep[][Section 2.4]{ES16}, we introduce truncated Poincar\'{e} 
series $P_{m,w,h}$ with $m\in \Z$, $w\in\R_{>0}$ and $h\in L^\dual/L$, of 
weight $k = 2-(p+q)$:
\[
\begin{gathered}
P_{m,w,h}(\tau) =  \frac{1}{2}\sum_{A\in \Gamma_\infty \backslash \SL_2(\Z)} 
\left[ \sigma_w({\tau})q^{-m}\ebase_h \right]\vert_{k,L^{-}} A,\\
\qquad\text{where}\quad
\sigma_w(\tau) = \begin{cases}
1 & \text{if}\quad v\geq w \\
0 & \text{if}\quad v < w.
\end{cases}
\end{gathered}
\]
Further, if $m\not\in \frac12 \hlf{h}{h} + {\Z}$ we set  $P_{m,w,h} = 0$. 
\begin{proposition} \label{prop:GK_Theta}
The regularised pairing $\pairregLm{ P_{m, w,h}}{\Theta(\cdot,z)}$ exists. 
On $\Dom\setminus \cycl{h,m}$, it satisfies the identity
\[
\pairregLm{ P_{m, w,h}}{\Theta(\cdot,z)}    = 
-\Xi^K(m,w,h) - \delta_{m,0} \delta_{h,0} \psi(0) \log(w). 
\]
The Kudla type Green form $\Xi^K(m,w,h)$ can thus be expressed as a 
regularized theta lifting.  This also affords an (albeit
discontinuous) extension of $\Xi^K(m,w,h)$ to all $\Dom$. 
\end{proposition}
\begin{proof}
Assume that $z\notin \cycl{m,h}$. We evaluate the regularized pairing by 
unfolding using the modularity of $\Theta$ and see
\[
\begin{aligned}
\pairregLm{P_{m,w,h}}{\Theta(\cdot,z)} & =
\CT_{s=0}\lim_{t\rightarrow \infty} \int_{\mathcal{F}_{t}-\mathcal{F}_{w}} 
\sum_{\substack{\lambda \in L + h \\ \hlf{\la}{\la} = m }}
q^{-m}\psikm(\sqrt{2v}\lambda) v^{-s} d\mu \\
& = \CT_{s=0} \int_{w}^\infty
\sum_{\substack{\lambda \in L + h \\ \hlf{\la}{\la} = m }}
\psikm^0(\sqrt{2v}\lambda,z)  v^{-s-1} d v.
\end{aligned}
\]
Now, for $m\neq 0$ this extends smoothly to the entire $s$-plane and for $s=0$, 
we obtain 
\[
 -\sum_{\substack{\lambda \in L + h \\ \hlf{\la}{\la} = m }}
\Psising^0(\sqrt{2v}\lambda,z) =- \Xi^K(m,w,h).
\]
Similarly, for $m = 0$ we obtain $-\Xi^K(m,w,h)$ from the sum over 
$\lambda\neq 0$. The term for $\lambda = 0$ contributes
\[
\psi(0)\CT_{s=0}\lim_{t\rightarrow \infty} \int_{w}^t v^{-s-1} dv = 
-\psi(0) \CT_{s=0}\lim_{t\rightarrow \infty} \tfrac{1}{s} \left( 
t^{-s} - w^{-s} \right)
=- \psi(0) \log(w). \qedhere
\] 
\end{proof}

\subsection{The difference of the two Green forms as a modular form}
Now, with the results of \cite{ES16}, we can show that the difference of 
$\Green^K(m,v)$ and $\Green^B(m)$ is, essentially a modular form. 

\begin{lemma}
The difference
\[
\pairregLm{ P_{m,w,h}}{\Theta(\cdot, z)} 
- \pairregLm{ F_{m,h}}{\Theta(\cdot,z)} 
\]
extends to a smooth differential $(q-1,q-1)$-form on $\Dom$. 
\end{lemma}
\begin{proof}
Since the principal part of $F_{m,h}$ is given by $q^{-m} \ebase_h$ this is immediate from Proposition~\ref{prop:singularities} and Proposition~\ref{prop:GK_Theta}.
\end{proof}

We now assume $p+q>2$. Using \citep[][Theorem 1.1]{ES16}, we show the following:
\begin{theorem}\label{ES-FH-Th} Assume $p+q>2$, and fix $z \in \Dom$. The generating series
\[
F(\tau,z) = -\log(v)\psi(0) \ebase_0 -
\sum_{m\in\Q} \left(  \Xi^K(m,v) - \Green^B(m) \right)(z)\, q^m 
\]
is an element of $\AweakL{p+q}{L}$. Furthermore, $F$ satisfies $L_{p+q}(F)(\tau,z) = - \Theta(\tau, z)$ and is orthogonal to cusp forms.
\end{theorem} 
\begin{proof}
We observe that $\Theta(\tau, z; \psi)$, as a function on $\mathbb{H}$ is 
contained in the space  $\AmodL{(p+q-2)}{L}$, see Definition \ref{def:Amod}.
Clearly by Proposition \ref{prop:GK_Theta} the generating series above can be written as 
\begin{align*}
\sum_{m\in \Q} \sum_{h \in L^\dual/L} \pairregLm{P_{m,v, h} - 
F_{m,h}}{\Theta_h(\cdot,z)} q^m \ebase_h .
\end{align*}
Since $\kappa$ is an integer and satisfies $\kappa = p+q-2>0$,  by \citep[][Theorem 1.1]{ES16},  this generating series, as a function on $\mathbb{H}$,  is the $q$-expansion of a modular form $F$ in $\AweakL{p+q}{L}$,  which satisfies $L_{p+q}(F) = - \Theta$,  has trivial principal part and  trivial cuspidal holomorphic projection, i.e.\  for every 
cusp form $G$ in $S_{\kappa, L}$, the (regularised) Petersson 
product $\left\langle F, G \right\rangle^{reg}$ vanishes. 
\end{proof} 

\begin{remark}
We note that Theorem~\ref{ES-FH-Th} also gives a different approach to the duality statement Theorem~\ref{B-KM-duality-Th}.
Namely, consider $dd^cF(\tau)$ and take the Petersson inner product with the holomorphic Poincare series $\xi_{k}(F_{m,h}(\tau, s_0))$ of index $(-m,h)$. This vanishes and computing the inner product explicitly (using the formulas for holomorphic projection) one obtains Theorem~\ref{B-KM-duality-Th}. We leave the details to the reader. 

We thank Stephan Ehlen for this comment.
\end{remark}

\section{Poincar\'e series}\label{sec:OTGreen}

In this section we introduce and study the form $\Green^B_{s} (m,h)$ depending on a complex parameter $s$ and identify it with the Green form constructed by Oda-Tsuzuki \cite{OT09}.

Namely, for $s \in \C$ with $\Re(s)=\sigma > 1$, we define 
\[
\Green^B_{s} (m,h)(z) \vcentcolon = 
\lim_{t\rightarrow \infty} \int_{ \mathcal{F}_t} 
\left\langle F_{m,h}(\tau,s), \Theta(\tau, z)\right\rangle_L\, d\mu
\]
Similar to Section 2.2 in \cite{Br02} it can be seen that the (regularized) integral converges for $\sigma$ sufficiently large and can be analytically continued to the region $\sigma >1$ with $s \ne s_0$. 

\begin{remark}
We can also define $\Green^B_{s_0} (m,h)(z)$ for $s=s_0$ as the constant term of the Laurent expansion of $ \Green^B_{s} (m,h)(z)$ at $s=s_0$. We note that $\Green^B(m,h)$, see \eqref{B-current} and $\Green^B_{s_0}(m,h)$ are not quite identical; due to the different regularization procedures, they differ by a smooth term. See \citep[][Proposition 2.11]{Br02} for further details in the orthogonal case.
\end{remark}

To ease the comparison with the work of Oda-Tsuzuki, we use the identification of differential forms on $\Dom$ with $K$-invariant functions on $G$ with values in $\wwedge{\bullet}\mathfrak{p}^{\ast}$. In our situation, this means to consider $\Green^B_{s} (m,h)$ as a function on $G$ with values in $\wwedge{q-1,q-1}\mathfrak{p}^{\ast}$ by first setting $\psi(x,g) := \psi(g^{-1}x,z_0)$ for $g \in G$ and then defining
\[
\Green^B_{s} (m,h)(g) \vcentcolon = 
\lim_{t\rightarrow \infty} \int_{ \mathcal{F}_t} 
\left\langle F_{m,h}(\tau,s), \Theta(\tau, g)\right\rangle_L\, d\mu. 
\]
It is then clear that $\Green^B_{s} (m,h)$ is holomorphic in $s$ in the convergent range.

\subsection{An eigenvalue equation}

Now, we show that the Green form  $\Green^B_s (m,h)$ 
satisfies an eigenvalue equation under the action of the Casimir element for $\Ug(p,q)$ as the one in \cite{OT09}, Theorem 18 (iii) (with a different normalization of the holomorphic parameter $s$). 
The overall strategy follows of \citep[][Chapter 4.1]{Br02} using results of Shintani \cite{Shin75} and additionally Hufler \cite{Huf17}. 
We denote by $\CSL$, $\CUpq$ and $\CO$ the respective Casimir 
elements of $\SL_2(\R)$, $\Ug(p,q)$ and $\Orth(2p, 2q)$ in the universal 
enveloping algebra. 

Let $\phi = \phi(x,\tau, z_0)$ be a Schwartz form and $\kappa$ the weight of 
$\phi(\tau)$ under the Weil representation. As $\phi$ satisfies condition 
(1.19) of \cite{Shin75} with $m = 2\kappa$, by \citep[][Lemma 1.4]{Shin75}  we 
have
\[
\begin{aligned}
\omega(g'_\tau)\CSL\, \phi(x) & = 
4\left[ v^2 \left(\frac{\partial^2}{\partial^2 u} + 
\frac{\partial^2}{\partial^2 v}   \right) - \kappa i v \frac{\partial}{\partial 
u} \right] \omega(g_\tau)\phi(x)  \\
&  =- 4\left[ \Delta_{\kappa}  - v \kappa \frac{\partial}{\partial v} \right]
 \omega(g_\tau)\phi(x),
\end{aligned}
\]
wherein $g'_\tau = \begin{psmallmatrix}
\sqrt{v} & u\sqrt{v}^{-1}\\ & \sqrt{v}^{-1}
\end{psmallmatrix}$. 
By a  brief calculation we thus have
\[
4\Delta_\kappa \phi(x, \tau)  = \kappa(\kappa -2) \phi(x, \tau)
 -v^{-\frac{\kappa}2 }\omega(g'_\tau)\CSL \, \phi(x) .
\]
Now, by \citep[][Lemma 1.5]{Shin75} we have with $m = \dim_\C(V) = p+ q$
\[
\CSL\, \phi(x)  = \left[ \CO + m (m - 2)\right] \phi(x).
\]
 We note that the operation of $\SL_2(\R)$ by
the Weil representation commutes with $\CO$. Hence, we get
\[
4\Delta_\kappa\phi(x, \tau)  = \left[ \kappa(\kappa -2) - m(m-2) \right]  
\phi(x, \tau)  -  \CO  \,\phi(x, \tau) . 
\]
Now, by a result of Hufler \citep[see][Satz 6.10]{Huf17}, who carries out the analogous computations for the Schwartz form $\varphi_0$,
\begin{equation}\label{eq:CUpqCO2p2q}
\CUpq\: \phi(x) 
 = 
\CO \phi(x)  - 2\left(
 \Im\left(\sum_{j = 1}^m z_j \frac{\partial}{\partial z_j}\right)\right)^2 
 \phi(x).
\end{equation}
Now set $\phi = \psikm$. The second term on the right hand side of  
\eqref{eq:CUpqCO2p2q} vanishes for $\psikm$  and with $\kappa = p + q - 2 = m 
-2$, we get 
\[
4\Delta_\kappa \psikm = -4\kappa \psikm - \CUpq \psikm.
\]
The following Lemma is an immediate consequence. 
\begin{lemma}\label{lemma:Theta_diffeq} The theta function $\Theta(\tau, z)$, 
considered as a function on 
$\mathbb{H}$, satisfies the following differential equation: 
\[
4 \Delta_\kappa \Theta(\tau, z_0) =  \bigl[ -4\kappa - \CUpq 
\bigr]\Theta(\tau, z_0). 
\]
\end{lemma}
Noting that the Poincar\'{e} series $F_{m,h}$ is an eigenfunction of $\Delta_k$ 
with eigenvalue $\tfrac{\kappa^2}{4} + \tfrac{\kappa}{2} + s(1-s)$ 
\citep[see][p.\ 29]{Br02}, we have the following analogue of \citep[][Lemma 
4.4]{Br02}, the proof of which is quite similar:
\begin{lemma}\label{lemma:Laplace_shift}
For the regularised pairing of $\Theta(\tau, z)$ and the Maass Poincar\'{e} 
series  
$F_{m,h}$ of weight $-\kappa$, we have 
\[
\begin{multlined}
\pairregLm{F_{m,h}}{\Delta_\kappa \Theta(\cdot, z)} = 
 \pairregLm{\Delta_{-\kappa}F_{h,m}}{\Theta(\cdot, z)}  - \kappa 
 \pairregLm{F_{h,m}}{ \Theta(\cdot, z))}\\
= \left(\tfrac{\kappa^2}{4} - \tfrac{\kappa}{2} + s(1-s)\right) 
\pairregLm{F_{h,m}}{\Theta(\cdot, z))}.
\end{multlined}
\]
\end{lemma}
By combining the two Lemmas we get 
\begin{theorem}\label{Th:Casimir}
Recall $ \kappa=p+q-2$. The Green form $\Green^B_s(h,m)$ is an eigenfunction of the Casimir 
operator $\CUpq$, with 
\begin{equation}
  \CUpq \Green^B_s(m,h) =\left( (2s-1)^2 - (\kappa + 1)^2\right) 
\Green^B_s(m,h).
\end{equation}
\end{theorem}
\begin{proof}
Due to locally uniform convergence of the regularized lift and all partial 
derivatives, we have
\[
\begin{aligned}
\CUpq & \pairregLm{F_{m,h}(\cdot, s)}{\Theta(\cdot, z)} = 
\pairregLm{F_{m,h}(\cdot, s)}{ \CUpq  \Theta(\cdot, z)} \\
& = -4\pairregLm{F_{m,h}(\cdot, s)}{(\Delta_\kappa \Theta)(\cdot, z)}
-4k \pairregLm{F_{m,h}(\cdot, s)}{\Theta(\cdot, z)},
\end{aligned}
\]
by Lemma \ref{lemma:Theta_diffeq}. The statement then follows by Lemma 
\ref{lemma:Laplace_shift}. 
\end{proof}

\subsection{Unfolding against the Poincar\'e series}

In this section, we calculate $\Green^B(m,h)(z_0)$ by unfolding the theta 
integral against the Poincar\'{e} series $F_{m,h}(\tau,s)$. To facilitate notation we write
\begin{equation}\label{eq:defPell}
 \mathbf{P}_{2\ell}^\psi(\lambda)\vcentcolon =  
\frac{2i (-1)^{q-1}}{2^{2(q-1)}}
\sum_{\underline{\alpha},\underline{\beta}}
 P_{\underline{\alpha},\underline{\beta}; 2\ell}^{2q-2}(\lambda) \otimes 
 \Omega_{q-1}(\underline{\alpha};\underline{\beta})
\end{equation}
for the homogeneous component of degree $2\ell$ of the polynomial part $P_\psi(\la)$ of the Schwartz form $\psi$. 

\begin{theorem}\label{G-formula}
We have
\[
\begin{multlined}
\Green^B_s(m,h)  =  \frac{\left(2\pi\abs{m} \right)^{s 
-\frac{k}{2}}}{2\Gamma(2s)} \\
\times \sum_{\substack{\lambda \in h + L \\ \hlf{\la}{\la} = m}} 
\sum_{\ell = 0 }^{q-1}
\mathbf{P}_{2\ell}^\psi(\lambda)
\frac{\Gamma(s-\tfrac{k}{2}+\ell)\;}{\left(2\pi\bigl(\lambda_{z_0^\perp}, \lambda_{z_0^\perp}\bigr)\right)^{s-\frac{k}{2}
		+ \ell}}
\,{_2}F_1\biggl(s-\tfrac{k}{2} + \ell, s + \tfrac{k}{2}; 2s; 
\frac{\abs{m}}{\bigl(\lambda_{z_0^\perp},\lambda_{z_0^\perp}\bigr) }\biggr).
\end{multlined}
\]
Here ${_2}F_1$ denotes the standard Gaussian hypergeometric function. 

\end{theorem}

\begin{proof}
From the definition of $F_{m,h}$ \eqref{eq:defFmh}, and using the unitarity of $\rho_L$ and the transformation property of $\Theta(\tau)$ we have
\[
\begin{aligned}
\Green^B_s(m,h) & 
\\
= \frac{1}{4\Gamma(2s)} & \int_{ \mathcal{F} }^{reg} 
\langle \sum_{A\in \Gamma_\infty \backslash \SL_2(\Z)}  
\mathcal{M}_s(4\pi \abs{m} \Im(A\tau)) e^{2\pi i m \Re(A\tau)} j(A,\tau)^{-k}\ebase_h, \rho_L(A)
\Theta (\tau, z_0) \rangle_{L^-} \,d\mu\\
=\frac{1}{4\Gamma(2s)} & \int_{ \mathcal{F} }^{reg} 
\sum_{A\in \Gamma_\infty \backslash \SL_2(\Z)}  
\mathcal{M}_s(4\pi \abs{m} \Im(A\tau)) e^{2\pi i m \Re(A\tau)}
\theta_h(A\tau, z_0) \,d\mu.
\end{aligned}
\]
Now, arguing exactly as in \citep[][p.55f]{Br02}, the unfolding (justified by 
absolute convergence for $\sigma > 1 + \frac{p}{2} + 
\frac{q}{2}$) is allowed, and we obtain
\[
\Green^B_s(m,h) = 
\frac{2}{4\Gamma(2s)} \int_{ v=0 }^\infty\int_{u=0}^1
\mathcal{M}_s(4\pi \abs{m} v) e^{2\pi i m u}
\theta_h(\tau, z_0) v^{-2}\, du\, dv.
\]
Inserting the Fourier expansion of $\theta_h(\tau, z)$ and 
integrating over $u$ one sees
\[
\begin{multlined}
\frac{\left(4\pi\abs{m} \right)^{-\frac{k}{2}}}{2\Gamma(2s)}\int_{v=0}^\infty 
\sum_{\substack{\lambda \in h + L \\ \hlf{\la}{\la} = -m}} 
 {M}_{-\frac{k}{2}, s - \frac12}(4\pi \abs{m} v)
e^{4\pi\hlf{\lambda_{z_0}}{\lambda_{z_0}}v - 2\pi\hlf{\lambda}{\lambda} v}
 v^{-\frac{k}{2} - 1} \sum_{\ell = 0}^{q-1} v^{\ell} 
 \mathbf{P}_{2\ell}^\psi(\sqrt{2}\lambda)\\
= \frac{\left(4\pi\abs{m} \right)^{-\frac{k}{2}}}{2\Gamma(2s)}\!\!
\sum_{\substack{\lambda \in h + L \\ \hlf{\la}{\la} = -m}} 
\sum_{\ell = 0}^{q-1} 2^{\ell} \mathbf{P}_{2\ell}^{\psi}(\lambda)
\int_{v=0}^\infty 
v^{-\frac{k}{2} +\ell - 1 } {M}_{-\frac{k}{2}, s - \frac12}(4\pi \abs{m} v)
e^{-2\pi v\hlf{\lambda}{\lambda}_{z_0}} dv.
\end{multlined}
\]
The integrals are Laplace transforms, which can be evaluated as usual 
\citep[see][p.\ 215]{EMOT54}. 
We get for each integral
\[
\frac{\left(4\pi\abs{m}\right)^{s}}{\left(4\pi\bigl(\lambda_{z_0^\perp}, \lambda_{z_0^\perp}\bigr)\right)^{s-\frac{k}{2}
		+ \ell}} \Gamma\left(s-\tfrac{k}{2}+\ell\right)
\,{_2}F_1\biggl(s-\tfrac{k}{2} + \ell, s + \tfrac{k}{2}; 2s; 
\frac{\abs{m}}{\bigl(\lambda_{z_0^\perp},\lambda_{z_0^\perp}\bigr) }\biggr),
\]
and the result follows. 
\end{proof}

We denote the individual summands for $\Green^B_s(h,m)$ in Theorem~\ref{G-formula} by $\phi_s(\la)$, that is, 
\[
\begin{multlined}
\phi_s(\la):=  \frac{\left(2\pi\abs{m} \right)^{s 
-\frac{k}{2}}}{2\Gamma(2s)} \sum_{\ell=0}^{q-1}
\mathbf{P}_{2\ell}^\psi(\lambda)
\frac{\Gamma(s-\tfrac{k}{2}+\ell)\;}{\left(2\pi\bigl(\lambda_{z_0^\perp}, \lambda_{z_0^\perp}\bigr)\right)^{s-\frac{k}{2}
		+ \ell}}
\,{_2}F_1\biggl(s-\tfrac{k}{2} + \ell, s + \tfrac{k}{2}; 2s; 
\tfrac{\abs{m}}{\bigl(\lambda_{z_0^\perp},\lambda_{z_0^\perp}\bigr) }\biggr).
\end{multlined}
\]

\begin{proposition}\label{O-T-Th18}
Assume $m>0$. Let $H$ be the stabilizer of $\la$ in $G$. Then 

\begin{itemize}

\item[(i)]
\[
\phi_s(\la) \in C^{\infty}\left((G-HK)/K; \wwedge{(q-1),(q-1)} \mathfrak{p}^{\ast}\right)
\]

\item[(ii)]
$ \phi_s(\la)$ is holomorphic in $s$.

\item[(iii)]
Let $\la = \sqrt{m} v_1$ and consider $g=a_t= \exp(tX_{1 p+q})$ as in the proof of Proposition~\ref{prop:localint}. Then there exists a non-zero constant $C$ such that 
\[
\lim_{t \to 0} t^{2(q-1)}  \phi_s(\la, a_t) = C \Omega_{q-1}(\underline{1},\underline{1}).
\]

\item[(iv)]
With the hypothesis as in (iv) we have
\[
 \phi_s(\la, a_t) = O(e^{-(2Re (s) + p+q)t)})
\]
as $ t \to \infty$.

\end{itemize}

\end{proposition}

\begin{proof}
(i) and (ii) are clear.
Now assume $\la = \sqrt{m} v_1$ and take $g=a_t= \exp(tX_{1 p+q})$. Then $a_t^{-1}\lambda_{z_0^\perp} =  \cosh(t) \sqrt{m} v_1$, and we calculate
\begin{align*}
&\phi_s(\la,a_t)  =\frac{1}{2\Gamma(2s)} 
\sum_{\ell = 0 }^{q-1}
\mathbf{P}_{2\ell}^\psi(\sqrt{m} v_1) \\
& \hphantom{\phi_s(\la,a_t)  =\frac{1}{2\Gamma(2s)} 
	\sum_{\ell = 0 }^{q-1}}
\times \quad
\frac{\Gamma(s-\tfrac{k}{2}+\ell)}{\left(2\pi m\right)^\ell \left(\cosh t\right)^{2s-k
		+ 2\ell}} 
\,{_2}F_1\Bigl(s-\tfrac{k}{2} + \ell, s + \tfrac{k}{2}; 2s; 
\frac{1}{\cosh^2t }\Bigr) \\
&=
\frac{1}{2\Gamma(2s)} 
\sum_{\ell = 0 }^{q-1}
\mathbf{P}_{2\ell}^\psi(  v_1)
\frac{\Gamma(s-\tfrac{k}{2}+\ell)}{\left(2\pi\right)^\ell \left(\cosh t\right)^{2s-k
 }}
\,{_2}F_1\Bigl(s-\tfrac{k}{2} + \ell, s + \tfrac{k}{2}; 2s; 
\frac{1}{\cosh^2t }\Bigr) \\
&=
\frac{1}{2\Gamma(2s)} 
\sum_{\ell = 0 }^{q-1}
\mathbf{P}_{2\ell}^\psi(  v_1)
\frac{\Gamma(s-\tfrac{k}{2}+\ell)}{\left(2\pi\right)^\ell \left(\cosh t\right)^{2s-k }}
\left (\frac{\sinh t}{\cosh t}\right)^{-2\ell} \,{_2}F_1\Bigl(s+\tfrac{k}{2} -\ell, s - \tfrac{k}{2}; 2s; 
\frac{1}{\cosh^2t }\Bigr).
\end{align*}
Here we used ${_2}F_1(a,b;c,z) = (1-z)^{c-a-b}{_2}F_1(c-a,b-a;c,z)$. Then (iii) follows from the second line of the previous equation, while (iv) from the third line, properties of $\mathbf{P}_{2\ell}^\psi(  v_1)$ and ${_2}F_1(s+\tfrac{k}{2} -(q-1), s - \tfrac{k}{2}; 2s; 1) = \Gamma(2s)\Gamma(q-1)/\Gamma(s-\tfrac{k}{2} +q-1)\Gamma(s+\tfrac{k}{2})$.
\end{proof}

Oda and Tsuzuki in \cite{OT09}, Theorem 18, show that the properties (i)-(iv) in Theorem~\ref{O-T-Th18} together with the Casimir equation uniquely determine the function. Using Theorem~\ref{Th:Casimir} we conclude 

\begin{corollary}
The Green forms $\Green^B_s(m,h)$ agree (up to a constant) with the (global) Green forms constructed by Oda  and Tsuzuki in \cite{OT09}.
\end{corollary}

\begin{remark}
Similarly one can evaluate the regularized pairing of $\Theta(\tau,z)$ with the 
non-holomorphic Eisenstein series 
\begin{equation*}
E_h(\tau,s ) = 
\sum_{A\in \Gamma_\infty \backslash \SL_2(\Z)} v^s\ebase_h \mid_{k,L^-} A, 
\end{equation*}
corresponding to $\Green^B_s(0,h)$. After unfolding, and integration  one has 
\[
\begin{gathered}
\pairregLm{E_h(\cdot,s )}{\Theta(\cdot, z)}\mid_{z=z_0} = 
 2  \sum_{\substack{\ell = 0 }}^{q-1} \frac{\Gamma(s + \ell)}{(2\pi)^{s+\ell}}
\sum_{\substack{\lambda \in L + h\\ \hlf{\la}{\la} = 0}}
\hlf{\lambda_{z_0}}{\lambda_{z_0}}^{-s -\ell }
\mathbf{P}^\psikm_{2\ell}(\lambda).
\end{gathered}
\]
This expression can be written in terms of Eisenstein series 
for the discriminant kernel $G(L)$ in $\Ug(V)$. After setting 
\[
\zeta_{h,\lambda}(s)\vcentcolon = 
\sum_{\substack{a\in\mathcal{O}_F^\times\\ a\lambda \in L + h}}
\operatorname{N}_{\mathbb{F}/\Q}(a)^{-s}, \qquad
P(L) = \{ {\lambda \in L^\dual;\, \lambda\:\text{primitive}, \hlf{\la}{\la} = 
0}\},
\]
where $\mathbb{F}$ is the underlying imaginary quadratic field, one obtains
\[
2  \sum_{\substack{\ell = 0 }}^{q-1} 
 \frac{\Gamma(s + \ell)}{\abs{\mathcal{O}_\mathbb{F}^\times}(2\pi)^{s+\ell}}
\sum_{\lambda \in G(L)\backslash P(L)}
\zeta_{h,\lambda}(s)\,
\mathbf{P}_{2\ell}^\psi(\lambda)
\sum_{\gamma \in G(L)_\lambda\backslash G(L)} 
\hlf{\lambda_{\gamma z_0}}{\lambda}^{-s -\ell }.
\]
\end{remark}

\appendix

\section{Calculations in the Fock  model}\label{sec:calc_Fock}
In this section, we prove the main properties of the Schwartz functions from section \ref{sec:Schwartz}. We use the polynomial Fock model for the Weil representation, the setup of which is reviewed in section \ref{sec:focksetup}.
We use the intertwining map $\iota : \mathcal{S}(V)\longrightarrow \mathcal{P}(\C^{2(p+q)})$ between the Schr\"{o}dinger model and the space of complex polynomials in $2(p+q)$ variables, on which the action of the Weil representation $\omega$ is given by the Fock model.  Note that $\iota(\varphi_0) = 1$. Further main properties of the intertwining operator are summarized in Lemma \ref{inter1}. 

We abbreviate the variables in the Fock model for $\Ug(p,q) \times \Ug(1,1)$ by $z_{\alpha}''=z_{\alpha 1}''$, $z_{\alpha}'=z_{\alpha 2}'$,  $z_{\mu}'=z_{\mu 1}'$ and  $z_{\mu}''=z_{\mu 2}''$. We then have (see Lemma \ref{inter1}):
\[
\Dhowe = 
\frac{1}{2^{2q}}\left(\frac{-i}{\sqrt{2}\pi}\right)^{q}
\prod_{\mu}\sum_{\alpha=1}^p z''_\alpha \otimes A'_{\alpha\mu} \quad \text{and}\quad
\bar \Dhowe = \frac{1}{2^{2q}}\left(\frac{-i}{\sqrt{2}\pi}\right)^{q} 
\prod_{\mu}\sum_{\beta = 1}^p z'_{\beta} \otimes A''_{\beta \mu}.
\]  
By applying this to $1\otimes 1 = \iota(\varphi_0 \otimes 1)$, we see that $\phikm$ is given by
\[
\begin{aligned}
\phikm  = 
 \frac{(-1)^q}{2^{3q}\pi^{2q}} \sum_{\substack{ \alpha_1, \dotsc, \alpha_q \\ \beta_1, \dotsc, \beta_q}}
z''_{\alpha_1} \dotsm z''_{\alpha_q} z'_{\beta_1}\dotsm z'_{\beta_q}\otimes\Omega_q(\alpha_1, \dotsc, \alpha_q; \beta_1, \dotsc, \beta_q),
\end{aligned}
\]
 while the form $\psikm$ is given by 
\begin{equation*}\label{eq:psikmFock_expl}
\psikm  = \frac{2i}{2^{3(q-1})\pi^{2(q-1)}}
 \sum_{\substack{\alpha_1,\dotsc,\alpha_{q-1}
		\\  \beta_1,\dotsc,\beta_{q-1}}}z''_{\alpha_1}\dotsm z''_{\alpha_{q-1}} z'_{\beta_1}\dotsm z'_{\beta_{q-1}} 
\otimes \Omega_{q-1}(\alpha_1,\dotsc,\alpha_{q-1}; \beta_1,\dotsc,\beta_{q-1}).
\end{equation*}

\subsection{Proof of Proposition~\ref{psi-properties}}

We first verify that $\psikm$ has the correct transformation behavior under the operation of $\frakk'\simeq \mathfrak{so}_2(\R)$.
\begin{lemma}\label{lemma:k_weights}
	Under the operation of $\frakk'$, the form $\psikm$ has weight $p+q-2$. That is, 
	\[ 
	\omega \left(\begin{smallmatrix}
	\phantom{-}0 & 1\\ -1 & 0
	\end{smallmatrix}\right) \psikm = i(p+q-2) \psikm.
	\]
	\end{lemma}
\begin{proof}
	 We use the formula for the operation of the generators of $\frakk'$ through the Weil representation from Lemma \ref{Fock2} on p.\ \pageref{Fock2}, setting $r=s=1$: 
	\begin{equation*}
	\begin{aligned} 
	& \omega(w_1\circ w_1 + iw_1\circ w_1 i)
	=  2i  \left[ \sum_{\alpha=1}^p {z}''_{\alpha} \frac
	{\partial}{\partial {z}''_{\alpha}} - \sum_{\mu' = p+1}^{p+q} {z}'_{\mu'} \frac{\partial}{\partial {z}'_{\mu'}} \right] + i(p-q) \\
	\text{and}\quad & \omega(w_2\circ w_2 - iw_2\circ w_2 i)  = \, 2i \left[ \sum_{\alpha'=1}^p {z}'_{\alpha'} \frac
	{\partial}{\partial {z}'_{\alpha'}} - \sum_{\mu = p+1}^{p+q} {z}''_{\mu} \frac{\partial}{\partial {z}''_{\mu}} \right] + i(p-q).
	\end{aligned}
	\end{equation*}
	Note that, since $\Phi_W(iw\circ w) = 0$, this is actually the same as $\omega(w_1\circ w_1)$ and $\omega(w_2\circ w_2)$, respectively.
	
	As $\mathfrak{su}(W) \simeq \mathfrak{sl}_2(\R)$, we are mainly interested in the behaviour of $\psikm$ under the operation of 
	$\left(\begin{smallmatrix}
	\phantom{-}0 & 1\\ -1 & 0
	\end{smallmatrix}\right)$ 
	(while of course, $\left(\begin{smallmatrix}
	i & 0\\ 0 & i \end{smallmatrix}\right)$ generates the center).
	We have
	\begin{equation*}
	\begin{aligned}
	\omega \left(\begin{smallmatrix}
	\phantom{-}0 & 1\\ -1 & 0
	\end{smallmatrix}\right)
	& = \omega\bigl( \tfrac12\left(w_1\circ w_1 + w_2\circ w_2 \right)\bigr) \\
	& = i \left[ \sum_{\alpha=1}^p {z}''_{\alpha} \frac
	{\partial}{\partial {z}''_{\alpha}} 
	+ \sum_{\alpha'=1}^p {z}'_{\alpha'} \frac
	{\partial}{\partial {z}'_{\alpha'}}
	- \sum_{\mu' = p+1}^{p+q} {z}'_{\mu'} \frac{\partial}{\partial {z}'_{\mu'}} 
	- \sum_{\mu = p+1}^{p+q} {z}''_{\mu}  \frac{\partial}{\partial {z}''_{\mu}} \right] + {i (p-q)}.
	\end{aligned}
	\end{equation*} 
	Bearing in mind that $\psikm$ doesn't depend on ${z}'_{\mu'}$ and ${z}''_{\mu}$ the claim now follows from 
	\[
	 \sum_{\alpha=1}^p {z}''_{\alpha} \frac
	{\partial}{\partial {z}''_{\alpha}} \psi =  \sum_{\alpha'=1}^p {z}'_{\alpha'} \frac
	{\partial}{\partial {z}'_{\alpha'}}
\psikm =(q-1)\psikm,
	\]
	which is easily checked. 
\end{proof}

\begin{lemma}\label{lemma:k_inv}
	The Schwartz form $\psikm$ is invariant under the operation of $\frakk$.
\end{lemma}
\begin{proof}
We need to show $Z (\psikm) =0$ for all $Z \in \frakk$. Using the explicit formula for $\psikm$ given above (and ignoring constants), this means, using that $Z$ acts as a derivation, 
\begin{align*}
0
& =  \sum_{\substack{\alpha_1,\dotsc,\alpha_{q-1}
		\\  \beta_1,\dotsc,\beta_{q-1} }} \omega(Z) \left(z''_{\alpha_1}\dotsm z''_{\alpha_{q-1}} z'_{\beta_1}\dotsm z'_{\beta_{q-1}} \right)\otimes  \Omega_{q-1}(\alpha_1,\dotsc,\alpha_{q-1};\beta_1,\dotsc,\beta_{q-1})	 \\
& \quad + \sum_{\substack{\alpha_1,\dotsc,\alpha_{q-1}
		\\  \beta1,\dotsc,\beta_{q-1} }} z''_{\alpha_1}\dotsm z''_{\alpha_{q-1}} z'_{\beta_1}\dotsm z'_{\beta_{q-1}}  \otimes Z.\left(\Omega_{q-1}(\alpha_1,\dotsc,\alpha_{q-1};\beta_1,\dotsc,\beta_{q-1})\right).
\end{align*}	

Now let $Z= Z'_{\alpha \beta} \in \Hom(V_+',V_+')$. Then the Weil representation action gives
\begin{align*}
& \omega(Z'_{\alpha\beta}) \left(z''_{\alpha_1}\dotsm z''_{\alpha_{q-1}} z'_{\beta_1}\dotsm z'_{\beta_{q-1}}\right) \otimes \Omega_{q-1}(\alpha_1,\dotsc,\alpha_{q-1};\beta_1,\dotsc,\beta_{q-1})\\
& = - \sum_{j=1}^{q-1} z''_{\alpha} z''_{\alpha_1}\dotsm \widehat{z''_{\alpha_j}}
\dotsm z''_{\alpha_{q-1}} z'_{\beta_1}\dotsm z'_{\beta_{q-1}}  \otimes \Omega_{q-1}(\alpha_1,\dotsc,\beta, \dotsc,\alpha_{q-1};\beta_1,\dotsc,\beta_{q-1})
\\ & \quad + 
 \sum_{j=1}^{q-1}  z''_{\alpha_1}
\dotsm z''_{\alpha_{q-1}} z'_{\beta} z'_{\beta_1}\dotsm   \widehat{z'_{\beta_{j}}} \dotsm z'_{\beta_{q-1}}  \otimes \Omega_{q-1}(\alpha_1,\dotsc,  \alpha_{q-1}; \beta_1,\dotsc, \alpha,\dotsc, \beta_{q-1}).
\end{align*}
Now $\frakk \simeq  \Hom(V_+',V_+')$ acts on $\mathfrak{p}_+ \simeq \Hom(V_-,V_+)$ by composition. We obtain
\[
 Z'_{\alpha \beta}. Z'_{\alpha_j \mu} = - \delta_{\beta \alpha_j} Z'_{\alpha \mu},
\]
and hence for the dual action we see
\[
Z'_{\alpha \beta}.\xi'_{\alpha_j \mu} = \delta_{\alpha \alpha_j} \xi'_{\beta \mu}.
\]
In the same way we see
\[
Z'_{\alpha \beta}.\xi''_{\beta_j \mu} = -\delta_{\beta \beta_j} \xi''_{\alpha \mu}.
\]
This gives
\begin{align*}
& z''_{\alpha_1}\dotsm z''_{\alpha_{q-1}} z'_{\beta_1}\dotsm z'_{\beta_{q-1}} \otimes  Z'_{\alpha \beta}. \Omega_{q-1}(\alpha_1,\dotsc,\alpha_{q-1};\beta_1,\dotsc,  \beta_{q-1}) \\
& = \sum_{j=1}^{q-1}  z''_{\alpha_1}\dotsm z''_{\alpha} \dotsm z''_{\alpha_{q-1}} z'_{\beta_1}\dotsm z'_{\beta_{q-1}} \Omega_{q-1}(\alpha_1,\dotsc, \beta, \dotsc \alpha_{q-1};\beta_1,\dotsc, \beta_{q-1}) \\
&  = -\sum_{j=1}^{q-1}  z''_{\alpha_1}\dotsm z''_{\alpha_{q-1}} z'_{\alpha_1'}\dotsm z'_\beta \dotsm z'_{\alpha_{q-1}'} \Omega_{q-1}(\alpha_1,\dotsc, \alpha_{q-1};\beta_1,\dotsc, \alpha, \dotsc \beta_{q-1}).
\end{align*}
Combining all this shows $Z'_{\alpha \beta} \psikm =0$, as desired. 

We now consider the action of $Z'_{\mu \nu} \in \Hom(V_-',V_-')$. The Weil representation action on $\psikm$ clearly vanishes. Now the action on $\mathfrak{p}^+$ is given by $Z'_{\mu \nu}. Z'_{\alpha \mu'} = \delta_{\mu \mu'} Z'_{\alpha \nu}$ and hence
\[
Z'_{\mu \nu} \xi'_{\alpha_j \mu'} = - \delta_{\nu \mu'}  \xi'_{\alpha_j \mu} \qquad \text{and} \qquad 
Z'_{\mu \nu} \xi''_{\beta_j \mu'} =  \delta_{\mu \mu'}  \xi''_{\beta_j \nu}.
\]
From this it is easy to see that 
\[
Z'_{\mu \nu} \Omega_{q-1}(\alpha_1,\dotsc, \alpha_{q-1};\beta_1,\dotsc,\beta_{q-1})=0.
\]
for all $\underline{\alpha},\underline{\beta}$. 
\end{proof}

\subsection{Proof of Theorem~\ref{prop:props_psikm}}

Recall
\[
d = \frac12\left(\partial + \bar{\partial}\right),\qquad d^c 
= \frac{\left( \partial - \bar\partial\right)}{4\pi i}, \qquad dd^c = -\frac{1}{4\pi i}\partial\bar{\partial}.
\] 
In the Fock model, the differential operators $\partial$, $\bar\partial$ are given by (see Lemmas \ref{Fock1}, \ref{Fock2})\label{lbl:DopsFock}
\begin{gather*}
\partial = 
\sum_{\alpha, \mu}  \left[ \frac{1}{4\pi} z''_{\alpha} z'_\mu
- 4\pi
\frac{\partial^2}{\partial z'_{\alpha}\partial z''_{\mu}} \right]
\otimes A'_{\alpha\mu}, \qquad 
\bar\partial =  
\sum_{\beta, \nu} \left[  \frac{1}{4\pi} z'_\beta z''_{\nu}
- 4\pi
\frac{\partial^2}{\partial z''_{\beta}\partial z'_{\nu}}\right] \otimes A''_{\beta\nu}. \\
\intertext{For the lowering operator $L = -\frac{i}{2} \,\omega(w_1 \circ w_2 + iw_1 \circ w_2 i)$, we have}
L = -4\pi\sum_{\gamma} \frac{\partial^2}{\partial z''_{\gamma} \partial z'_{\gamma} }
+ \frac{1}{4\pi} \sum_\mu z''_\mu z'_\mu. 
\end{gather*}

For simplicity we drop all constants and consider 
\begin{align*}
\phikm' & = 
\!\!
\sum_{\substack{ \alpha_1, \dotsc, \alpha_q \\ \beta_1, \dotsc, \beta_q}}\!
z''_{\alpha_1} \dotsm z''_{\alpha_q} z'_{\beta_1}\dotsm z'_{\beta_q}\otimes \xi'_{\alpha_1 p+1}\wedge 
 \dotsm \xi'_{\alpha_q p+q}
\wedge \xi''_{\beta_1 p+1}\wedge  \dotsm \wedge \xi''_{\beta_q p+q}, \\
\psikm'  &= 
 \sum_{\substack{\alpha_1,\dotsc,\alpha_{q-1}
		\\  \beta_1,\dotsc,\beta_{q-1}}}z''_{\alpha_1}\dotsm z''_{\alpha_{q-1}} z'_{\beta_1}\dotsm z'_{\beta_{q-1}} \\
& \quad \qquad \otimes  \sum_{j=1}^q \xi'_{\alpha_1 p+1} \wedge \dotsb \wedge \widehat{ \xi'_{\cdot p+j}} \dotsb  \wedge \xi'_{\alpha_{q-1} p+q} \wedge  \xi''_{{\beta_1} p+1} \wedge \dotsb \wedge \widehat{\xi''_{\cdot 
      p+j}} \wedge \dotsb \wedge \xi''_{{\beta_{q-1}} p+q}.
\end{align*}
Then the claim is equivalent to
\[
L \phikm' = (-1)^{q-1} 4\pi \partial \bar\partial \psikm',
\]
which we show by a direct calculation of both sides. We have
\begin{align*}
L \phikm' 
&= \frac1{4\pi} \left( \sum_{\mu} z_{\mu}''z_{\mu}' \right) \phikm' \\
& \; - 4\pi \sum_{\underline{\alpha},\underline{\beta}} \sum_{j,k=1}^q \delta_{\alpha_j \beta_k} 
z''_{\alpha_1}\dotsm \widehat{z''_{\alpha_j}} \dotsm z''_{\alpha_{q-1}} z'_{\beta_1}\dotsm \widehat{z'_{\beta_k}} \dotsm z'_{\beta_{q-1}} \\
& \qquad \qquad \otimes
\xi'_{\alpha_1 p+1}\wedge \dotsm \wedge \xi'_{\alpha_j p+j} \wedge 
 \dotsm \xi'_{\alpha_q p+q}
\wedge \xi''_{\beta_1 p+1}\wedge  \dotsm \wedge \xi''_{\beta_k p+k} \wedge \dotsm \wedge \xi''_{\beta_q p+q}.
\end{align*}
On the other hand,
\begin{align*}
\partial \bar\partial \psikm' &= \frac1{16 \pi^2} \sum_{\alpha,\beta,\mu,\nu} \left( z''_{\alpha}z'_{\beta} z'_\nu z''_\mu \otimes \xi'_{\alpha \nu} \wedge \xi''_{\beta \mu}\right) \;  \psikm' \\
&  - \sum_{\substack{\alpha_1,\dots,\alpha_{q-1} \\ \beta_1,\dots, \beta_{q-1} \\ \alpha,\beta,\mu} }z''_{\alpha_1}\dotsm z''_{\alpha_{q-1}}  \frac{\partial}{\partial z'_{\alpha}} \left (z'_{\beta} z'_{\beta_1}\dotsm z'_{\beta_{q-1}} \right) \\
& \quad \otimes \xi'_{\alpha \mu} \wedge \xi''_{\beta \mu} \wedge
\sum_{j=1}^q \xi'_{\alpha_1 p+1} \wedge \dotsb \widehat{ \xi'_{\cdot p+j}}  \dotsb  \wedge \xi'_{\alpha_{q-1} p+q} \wedge  \xi''_{{\beta_1} p+1} \wedge \dotsb \widehat{\xi''_{\cdot 
      p+j}} \dotsb \wedge \xi''_{{\beta_{q-1}} p+q}.
\end{align*}
For the first term, it is easy to see that only the terms $\mu=\nu$ contribute and one obtains
\[
(-1)^{q-1}  \frac1{16 \pi^2} \left( \sum_{\mu} z''_{\mu} z'_{\mu} \right) \phikm'.
\]
For the second, only terms $\mu = p+j$ contribute and one obtains
\begin{align*}
(-1)^q  & \sum_{\substack{\alpha_1,\dots,\alpha_{q-1} \\ \beta_1,\dots, \beta_{q-1} \\ \alpha_0,\beta_0} } 
z''_{\alpha_1}\dotsm z''_{\alpha_{q-1}} 
\sum_{k=0}^{q-1} \delta_{\alpha_0 \beta_k} z'_{\beta_0} z'_{\beta_1}\dotsm \widehat{z_{\beta_k}} \dotsm z'_{\beta_{q-1}} \\
& \otimes 
\sum_{j=1}^q \xi'_{\alpha_1 p+1} \wedge \dotsb \wedge{ \xi'_{\alpha_0 p+j}} \dotsb  \wedge \xi'_{\alpha_{q-1} p+q} \wedge  \xi''_{{\beta_1} p+1} \wedge \dotsb \wedge {\xi''_{\beta_0
      p+j}} \wedge \dotsb \wedge \xi''_{{\beta_{q-1}} p+q}.
\end{align*}
Now comparing the formulas for $L \phikm'$ and $ \partial \bar\partial \psikm'$ gives the claim.

\subsection{The auxiliary form \texorpdfstring{$d^c \psikm$}{dc psi}}

We now give a more explicit description of $d^c \psikm$. 
We have 
\begin{align*}
\frac1{4\pi}\partial\psikm &  = 
\frac{i}{2^{3(q-1)} \pi^{2(q-1)}}\frac{1}{2\pi}
 \sum_{\substack{\underline{\alpha},\underline{\beta} \\ \gamma, \mu}}
z''_{\gamma} z'_\mu z''_{\underline{\alpha}}z'_{\underline{\beta}} \otimes \xi'_{\gamma \mu} \wedge \Omega_{q-1}(\underline{\alpha};\underline{\beta}) \\
& = \frac{i}{2^{3q  -2}\pi^{2q-1}} 
  \sum_{\substack{\gamma,\alpha_1,\dots, \alpha_{q-1} \\ \beta_1,\dots,\beta_{q-1}}}
  z''_{\gamma} z''_{\alpha_1} \cdots z''_{\alpha_{q-1}} z'_{\underline{\beta}} \; \sum_{j=1}^q (-1)^{j-1} z'_{p+j}  \\
  &  \quad \otimes 
   \xi'_{\alpha_1 p+1} \wedge \dotsb \xi'_{\gamma p+j} \dotsb  \wedge \xi'_{\alpha_{q-1} p+q} \wedge  \xi''_{\beta_1 p+1} \wedge \dotsb \widehat{\xi''_{\cdot 
      p+j}} \dotsb \wedge \xi''_{\beta_{q-1} p+q}.
 \end{align*}
 Similarly, 
\begin{align*}
\frac1{4\pi}\bar \partial\psikm &  = \frac{i}{2^{3q  -2}\pi^{2q-1}} 
  \sum_{\substack{\alpha_1,\dots, \alpha_{q-1} \\ \gamma, \beta_1,\dots,\beta_{q-1}}}
  z''_{\underline{\alpha}} z'_{\gamma} z'_{\beta_1} \cdots z'_{\beta_{q-1}} \; \sum_{j=1}^q (-1)^{q+j} z''_{p+j}  \\
  &  \quad \otimes 
   \xi'_{\alpha_1 p+1} \wedge \dotsb \widehat{\xi'_{\cdot  p+j}} \dotsb  \wedge \xi'_{\alpha_{q-1} p+q} \wedge  \xi''_{\beta_1 p+1} \wedge \dotsb \xi''_{\gamma
      p+j} \dotsb \wedge \xi''_{\beta_{q-1} p+q}.
 \end{align*}
Now, $\psidc$ is the difference of these two terms. 

Finally, we want give an explicit form of $d^c \psikm$ in the Schr\"{o}dinger model. (Note that $\mathcal{D}_\mu \varphi_0 = 2\bar z_\mu \varphi_0$.) 
We have
\begin{equation} \label{eq:psidc_schroe}
d^c \psikm(x) = 
\frac{1}{2^{3q-1}\pi^{2q-1}} \biggl[
\sum_{\substack{\underline{\alpha},\underline{\beta} \\ \gamma }}
\mathcal{D}_{\underline{\alpha}} 
\mathcal{D}_{\gamma} 
\bar{\mathcal{D}}_{\underline{\beta}} 
\varphi_0(x) \otimes 
Q'_{\underline{\alpha},\gamma; \underline{\beta}}(x) 
- \sum_{\substack{\underline{\alpha},\underline{\beta} \\ \gamma }} \mathcal{D}_{\underline{\alpha}} \bar{\mathcal{D}}_{\gamma} \bar{\mathcal{D}}_{\underline{\beta}} \varphi_0(x) \otimes 
Q''_{\underline{\alpha}; \underline{\beta},\gamma}(x)\biggr]. 
\end{equation}
Here $Q'_{\underline{\alpha},\gamma; \underline{\beta}}(x)$ and $Q''_{\underline{\alpha}; \underline{\beta},\gamma}(x)$ are given by 
\begin{align*}
& Q'_{\underline{\alpha},\gamma; \underline{\beta}}(x) \\
& = 
\sum_{j=1}^q (-1)^{j-1} z_{p+j}  \otimes 
   \xi'_{\alpha_1 p+1} \wedge \dotsb \xi'_{\gamma p+j} \dotsb  \wedge \xi'_{\alpha_{q-1} p+q} \wedge  \xi''_{\beta_1 p+1} \wedge \dotsb \widehat{\xi''_{\cdot 
      p+j}} \dotsb \wedge \xi''_{\beta_{q-1} p+q} \\
 & Q''_{\underline{\alpha}; \underline{\beta},\gamma}(x) \\    
  & =
      \sum_{j=1}^q (-1)^{q+j} \bar{z}_{p+j}   \otimes 
   \xi'_{\alpha_1 p+1} \wedge \dotsb \widehat{\xi'_{\cdot  p+j}} \dotsb  \wedge \xi'_{\alpha_{q-1} p+q} \wedge  \xi''_{\beta_1 p+1} \wedge \dotsb \xi''_{\gamma
      p+j} \dotsb \wedge \xi''_{\beta_{q-1} p+q}.
\end{align*}

\section{The Fock model for unitary dual pairs}\label{sec:focksetup}

We review the Fock model of the Weil representation for the dual pair
$\Ug(p,q)\times \Ug(r,s)$. We follow \cite{A07,KM90,KM86}, see also \cite{FMcoeff}.

\subsection{The Fock model for the symplectic group}

Let $\left(\W, \sform{\cdot}{\cdot}\right )$ be a non-degenerate real
symplectic space of dimension $2N$ and let $J$ be a positive definite
complex structure on $W$, i.e., the bilinear form given by
$\sform{w_1}{Jw_2}$ is positive definite. Let
$e_1,\dots,e_N;f_1,\dots,f_N$ be a standard symplectic basis of $W$ so
that $J e_j=f_j$ and $J f_j = -e_j$. We decompose
\[
  \W \otimes \C = \W' \oplus \W''
\]
into the $+i$ and $-i$ eigenspaces under $J$.  Then $w_j' = e_j-f_ji$
and $w_j'' = e_j +f_ji$ form a (symplectic) basis for $\W'$ and $\W''$
respectively with $ \sform{w'_j}{w''_k} = 2i \delta_{jk}$.

We identify $\Sym^{\bullet}(\W'')$ with the polynomial functions
$\mathcal{P}(\C^N)= \C[z_1,\dots,z_N]$ on $\W'$ via
$z_j(w''_k) = \sform{w'_j}{w''_k} = 2i \delta_{jk}$. For
$\lambda \in \C^{\ast}$, we define an action $\rho_{\lambda}$ of $\W$
on $\mathcal{P}(\C^N)$ by
\[
  \rho_{\lambda}(w''_j) = z_j \qquad \text{and} \qquad
  \rho_{\lambda}(w'_j) = 2 i \lambda \frac{\partial}{\partial z_j},
\]
which induces an action of the associated quantum algebra
$\mathcal{\W}_{\lambda}$. We identify $Sym^2(\W)$ with
$\mathfrak{sp}(\W)$ via
\[
  (x \circ y)(z) = \sform{x}{z} y + \sform{y}{z}x.
\]
Then the action $\omega_{\lambda}$ of $\mathfrak{sp}(\W) \otimes \C$
on $\mathcal{P}(\C^N)$ is given by
\begin{equation}\label{eq:fock_lambda}
  \omega_{\lambda}(x \circ y) = \frac{1}{2\lambda} \left( \rho_{\lambda}(x)\rho_{\lambda}(y) + \rho_{\lambda}(y)\rho_{\lambda}(x) \right).
\end{equation}
This is the Fock model of the Weil representation with central
character $\lambda$.

\subsection{Unitary dual pairs}\label{subsec:pairs}

Let $\left(V,\hlfempty\right)$ be a complex vector
space of dimension $m=p+q$ with a non-degenerate Hermitian form
$\hlfempty$ of signature $(p,q)$. Recall that we assume that
$\hlfempty$ is $\C$-linear in the second and $\C$-antilinear in the
first variable. We pick standard orthogonal basis elements
$v_{\alpha}$ ($\alpha=1,\dots,p$) and $v_{\mu}$ ($\mu=p+1,\dots, m$)
of length $1$ and $-1$ respectively. We let $\theta$ be the Cartan
involution with respect to this chosen basis of $V$ and obtain a
decomposition $V = V_+ \oplus V_-$. Let
$G=\Ug(V) \simeq \Ug(p,q)$ and let
$\frakg_0 =\fraku(V) \simeq \fraku(p,q)$. We write
$\frakg =\frakg_0 \otimes \C$ for the complexification of $\frakg_0$,
viewed as a right $\C$ vector space.

We let $\left( W,\sform{}{} \right)$ be a complex
vector space with a non-degenerate skew-Hermitian form $\sform{}{}$ of signature $(r,s)$, again $\C$-linear in the
second and $\C$-antilinear in the first variable. We pick an ``orthogonal" basis $w_a$ ($a=1,\dots,r$) and $w_u$  ($u=r+1,\dots,r+s$) such that $\langle w_a,w_a\rangle = i$ and $\langle w_u,w_u\rangle = -i$. 
We obtain a decomposition $W= W_+ \oplus W_-$. We let  $J_0$ be a positive define complex structure with respect to this decomposition of $W$; that is, $J_0$ acts by multiplication with $-i$ on $W_+$ and with $i$ on $W_-$. We let $G'=\Ug(W) \simeq \Ug(r,s)$ and let 
$\frakg'_0 = \fraku(W) \simeq \fraku(r,s)$. 
The two factors of the maximal compact subgroup $K' \simeq \Ug(r)\times \Ug(s)$ of $G'$ act on the subspaces $W_+$ and $W_-$ respectively. We let $W_{\C} = W \otimes_{\R} \C$ be the complexification of $W$, which we again view as a right $\C$-vectorspace. We write $\frakg'= \frakg'_0 \otimes \C$. Then the $+i$-eigenspace $W'$ and the $-i$ eigenspace
$W''$ of ${J_0}$ are spanned by 
\begin{align*}
w_{a}'&\vcentcolon= w_a + iw_ai  \qquad  w_u'\vcentcolon= w_u-iw_ui, \\
w_{a}''&\vcentcolon= w_a-iw_ai  \qquad  w_u''\vcentcolon=  w_u+iw_ui.
\end{align*}
Similarly as for $V$, we define the $\R$-linear surjective map
by 
\[
\phi_W(w \circ \tilde{w})(x) = \sform{w}{x}\tilde{w} + \sform{\tilde{w}}{x}w.
\]
Note that we have $\phi_W(iw \circ \tilde{w}) = \phi_W(w \circ -i\tilde{w})$.
As for $V$, we will drop $\phi_W$ and just write ${w} \circ \tilde{w} \in \fraku(W)$. In the standard decomposition
\[
\frakg' = \frakk' \oplus \frakp'^+ \oplus \frakp'^-
\]
we have
\begin{align*}
\frakk'\phantom{+} & = \operatorname{span}_{\C} \{   w_a \circ w_b +  iw_a\circ w_bi  \} \oplus  \operatorname{span}_{\C} \{   w_u \circ w_v +  iw_u\circ w_vi  \}, \\
\frakp'^+ & =  \operatorname{span}_{\C} \{w_a \circ w_u -  iw_a \circ w_ui \}, \\
\frakp'^- & =  \operatorname{span}_{\C} \{w_a \circ w_u + iw_a \circ w_ui \}.
\end{align*}

Note that for $r=s=n$, $W$ is the split skew-Hermitian space of dimension $2n$, and we obtain a (symplectic) basis $e_j\vcentcolon=w_j+w_{n+j}$, $f_j\vcentcolon=-i(w_j-w_{n+j})$ of $W$. 
(Note that $\sform{e_j}{f_j} = 2$). Then the positive definite almost complex structure is the usual one given by $J_0 e_j = f_j$ and $J_0 f_j = -e_j$. 
For $n=1$, we have $\mathfrak{su}(W) \simeq \mathfrak{sl}_2(\R)$, and the isomorphism is realized by switching to the symplectic basis above. Then for $\frakk'$ we see
\[
\label{eq:genSL2}
\tfrac12(w_1 \circ w_1 + w_2 \circ w_2)=  \left( \begin{smallmatrix}\phantom{-} 0 &1 \\ -1 & 0 \end{smallmatrix} \right)
\quad\text{and}\quad \tfrac12(w_1 \circ w_1 - w_2 \circ w_2)=  \left( \begin{smallmatrix}i &0 \\ 0 & i \end{smallmatrix} \right).
\]
Note that $\phi_W(iw \circ w) =0$. Furthermore,
\begin{align*}
 L \vcentcolon&=  \tfrac12 \left( \begin{smallmatrix} \phantom{-} 1 & -i \\ -i & -1
\end{smallmatrix} \right) = \tfrac{-i}{2} [w_1 \circ w_2 + iw_1 \circ w_2i] \in \frakp'^- ,\\
R \vcentcolon& =  \tfrac12 \left( \begin{smallmatrix}
1 & \phantom{-}i \\ i & -1
\end{smallmatrix} \right) = \tfrac{i}{2}  [w_1 \circ w_2 - iw_1 \circ w_2i] \in \frakp'^+
\end{align*}
give rise to the classical Maass lowering and raising operators for $\SL_2$.

\paragraph{Pairing up} We define $\W = V \otimes_{\C} W $, which we consider as a real vector space of dimension $2(r+s)m$. We define a symplectic form on $\W$ by
\[
\LL v\otimes w,\tilde{v}\otimes \tilde{w}\RR =
\Re \hlf{v}{\tilde{v}}\sform{w}{\tilde{w}}.
\]
We note that $v_{\alpha} \otimes w_a$,  $v_{\mu} \otimes w_a$, $v_{\alpha} \otimes w_u$, $v_{\mu} \otimes w_u$ and $v_{\alpha} \otimes -iw_a$, $v_{\mu} \otimes iw_a$, $v_{\alpha} \otimes iw_u$,$v_{\mu} \otimes -iw_u$ span Langrangian subspaces and give rise to a symplectic basis of $\W$. (Note that $\LL v_{\alpha} \otimes w_a,v_{\alpha} \otimes -iw_a \RR =  1$). Now $J = \theta \otimes J_0$ defines a positive definite complex structure on $\W$. We let $\W_{\C} = \W \otimes_{\R} \C$ be the complexification of $\W$, which again we view as a right $\C$-vector space, and we extend $\LL\,,\,\RR$ $\C$-linearly. Then for the  $+i$-eigenspace $\W'$ and the $-i$ eigenspace
$\W''$ of ${J}$, we have
\begin{align*}
\W'  &=\operatorname{span}_{\C}\{v_{\alpha} \otimes w'_a,v_{\mu} \otimes w_a'', v_{\alpha} \otimes w'_u,v_{\mu} \otimes w_u''\}, \\
\W'' &=\operatorname{span}_{\C}\{v_{\alpha} \otimes w''_a,v_{\mu} \otimes w_a',v_{\alpha} \otimes w''_u ,v_{\mu} \otimes w_u'\}.
\end{align*}
Note $\LL v_{\alpha} \otimes w'_a,v_{\alpha} \otimes w''_a \RR = 2i$. In the Fock model, $\mathfrak{sp}(\W)$ acts on $\Sym^{\bullet}(\W'')$, which we identify $\Sym^{\bullet}(\W'')$ with $\mathcal{P}(\C^{(r+s)m})$ as follows. We denote the variables in $\mathcal{P}(\C^{(r+s)m})$ by $z''_{\alpha a}$ corresponding to
$v_{\alpha} \otimes w_a''$,  $z'_{\mu a}$  corresponding to  $v_{\mu} \otimes
w_a'$, $z'_{\alpha u}$ corresponding to
$v_{\alpha} \otimes w_u''$, and  $z''_{\mu u}$  corresponding to  $v_{\mu} \otimes w_u'$. 
Thus we have
\begin{align*}	
\rho_{\lambda}( v_{\alpha} \otimes w'_a) &= 2i\lambda \tfrac{\partial}{\partial {z}''_{\alpha a}},  &
\rho_{\lambda}(v_{\alpha} \otimes w'_u) &= 2i\lambda \tfrac{\partial}{\partial {z}'_{\alpha u}}, \\
\rho_{\lambda} (v_{\alpha} \otimes w_a'') & = {z}''_{\alpha a}, &
\rho_{\lambda} (v_{\alpha} \otimes w_{u}'') & ={z}'_{\alpha u}, \\
\rho_{\lambda}(v_{\mu } \otimes w''_{a}) &=  2i\lambda \tfrac{\partial}{\partial {z}'_{\mu a}}, &
\rho_{\lambda}( v_{\mu } \otimes w''_{u} ) &=  2i\lambda \tfrac{\partial}{\partial {z}''_{\mu u}}, \\
\rho_{\lambda} ( v_{\mu} \otimes w_{a}') & = {z}'_{\mu a}, &
\rho_{\lambda} ( v_{\mu} \otimes w_{u}') & =  {z}''_{\mu u}.
\end{align*}

\paragraph{Weil representation}
We naturally have $\fraku(V) \times \fraku(W) \subset \mathfrak{sp}
(V \otimes W)$, and one easily checks that the inclusions $j_1: \fraku(V) \to  \mathfrak{sp}(V \otimes W) \simeq \Sym^2_\R(V \otimes W)
$ and $j_2: \fraku(W)  \to  \mathfrak{sp}(V \otimes W) \simeq \Sym_\R^2(V
\otimes W)$ are given by 
\begin{multline*}
j_1( v \wedge \tilde{v} )  =\sum_{a=1}^r\Bigl[ (v
\otimes i w_a) \circ (\tilde{v} \otimes w_a) - (v \otimes w_a) \circ (\tilde{v} \otimes iw_a) \Bigr] \\
-  \sum_{u=r+1}^{r+s} \Bigl[(v \otimes i w_u)
\circ (\tilde{v} \otimes w_u) - (v \otimes w_u)
\circ (\tilde{v} \otimes iw_u)\Bigr]
\end{multline*}
and
\begin{multline*}
j_2( w \circ \tilde{w}) = \sum_{\alpha=1}^p \Bigl[(v_{\alpha} \otimes w)
\circ (v_{\alpha}\otimes \tilde{w})  +  (iv_{\alpha} \otimes w)
\circ (iv_{\alpha}\otimes \tilde{w}) \Bigr]\\
-  \sum_{\mu =p+1}^{p+q} \Bigl[(v_{\mu} \otimes w)
\circ (v_{\mu}\otimes \tilde{w})  + (iv_{\mu} \otimes w)
\circ (iv_{\mu}\otimes \tilde{w})\Bigr].
\end{multline*}
with $v,\tilde{v} \in V$ and $w,\tilde{w} \in W$. From this, we see
\begin{align*}
j_1\left( (v \wedge \tilde{v}) +
(iv \wedge \tilde{v}) i \right) & = \frac{1}{i} \sum_{a=1}^r [v
\otimes w_a'] \circ [ \tilde{v} \otimes w_a'' ]  
-  \frac{1}{i}  \sum_{u=r+1}^{r+s} [v
\otimes w_{u}'' ] \circ [ \tilde{v} \otimes w_u'], \\
j_1\left( (v \wedge \tilde{v}) -
(iv \wedge \tilde{v}) i \right)  &= -\frac{1}{i} \sum_{a=1}^r [v
\otimes w_a''] \circ [ \tilde{v} \otimes w_a' ]  
+  \frac{1}{i}  \sum_{u=r+1}^{r+s} [v
\otimes w_{u}' ] \circ [ \tilde{v} \otimes w_u'']
\end{align*}
and
\begin{multline*}
j_2( w \circ \tilde{w} \pm  (iw \circ \tilde{w})i) = \sum_{\alpha=1}^p
[v_{\alpha} \otimes (w \pm  iwi)] \circ
[v_{\alpha}\otimes (\tilde{w} \mp  i \tilde{w}i)] \\
-  \sum_{\mu =p+1}^{p+q}  [v_{\mu} \otimes (w \pm  i wi)] \circ
[v_{\mu}\otimes (\tilde{w} \mp  i\tilde{w}i)].
\end{multline*}

With this we obtain the formulas for the Weil representation (see p.\ \pageref{eq:fock_lambda}).
\begin{lemma}\label{Fock1}
	
	For the action of $\mathfrak{g} \simeq \mathfrak{u}(p,q)(\C)$ on $\mathcal{P}(\C^{2mn})$, we have the following:
	\begin{itemize}
		\item[(i)]
		The elements $Z'_{\alpha\beta}$, ${Z}''_{\alpha\beta}$ and
		$Z'_{\mu\nu}$, ${Z}''_{\mu\nu}$ in $\mathfrak{k}$ act by
		\begin{align*}
		\omega_\lambda(Z'_{\alpha\beta}) & = - \omega_\lambda({Z}''_{\beta \alpha}) = -  \sum_{a=1}^r
		z''_{\alpha a} \frac{\partial}{\partial {z}''_{\beta a}}
		+ \sum_{u=r+1}^{r+s}{z}'_{\beta u} \frac{\partial}
		{\partial {z}'_{\alpha u}} - \frac{r-s}{2}\delta_{\alpha\beta}, \\
		\omega_\lambda(Z'_{\mu\nu })& =-\omega_\lambda({Z}''_{\nu\mu })  =
		-\sum_{a=1}^r {z}'_{\nu a} \frac{\partial}{\partial {z}'_{\mu a}}   + \sum_{u=r+1}^{r+s} {z}''_{\mu u} \frac
		{\partial}{\partial {z}''_{\nu u}} - \frac{r-s}{2}\delta_{\mu\nu}.
		\end{align*}
		
		\item[(ii)]
		The elements $Z'_{\alpha\mu}$ of $\mathfrak{p^+}$ and  ${Z}''_{\alpha\mu}$ of 
		$\mathfrak{p^-}$ act by 
		
		\begin{align*}
		\omega_\lambda(Z'_{\alpha\mu}) &  = - \frac{1}{2i\lambda } \sum_{a=1}^r {z}''_{\alpha a}{z}'_{\mu a} + 2i\lambda   \sum_{u=r+1}^{r+s} \frac{\partial^2}{\partial
			{z}'_{\alpha u} \partial {z}''_{\mu u}} ,
		\\
		\omega_\lambda({Z}''_{\alpha\mu}) &  = 2i\lambda   \sum_{a=1}^r \frac{\partial^2}{\partial
			{z}''_{\alpha a} \partial {z}'_{\mu a}} - \frac{1}{2i\lambda } \sum_{u=r+1}^{r+s} {z}'_{\alpha u}{z}''_{\mu u}.
		\end{align*}
	\end{itemize}
\end{lemma}

\begin{lemma}\label{Fock2}
	For the action of $\frakg' \simeq \fraku(r,s)(\C)$ on $\mathcal{P}(\C^{2mn})$, we have the following:
\begin{itemize}
	\item[(i)] For $\mathfrak{k}'$ we have
	\begin{align*}
	\omega_\lambda(w_a \circ w_b + iw_a \circ w_bi) &= 2i \left[ \sum_{\alpha=1}^p {z}''_{\alpha b} \frac
	{\partial}{\partial {z}''_{\alpha a}} - \sum_{\mu = p+1}^{p+q} {z}'_{\mu a} \frac{\partial}{\partial {z}'_{\mu b}} \right] + i(p-q) \delta_{ab}, \\
	\omega_\lambda(w_u \circ w_v + iw_u \circ w_vi) &= 2i \left[ \sum_{\alpha=1}^p {z}'_{\alpha u} \frac
	{\partial}{\partial {z}'_{\alpha v}} - \sum_{\mu = p+1}^{p+q} {z}''_{\mu v} \frac{\partial}{\partial {z}''_{\mu u}} \right] + i(p-q) \delta_{uv}.
	\end{align*}
	\item[(ii)] For $\mathfrak{p}'^{\pm}$ we have
	\begin{align*}
	\omega_\lambda(w_a \circ w_u - iw_a\circ w_ui) &=  \frac1{\lambda}  \sum_{\alpha=1}^p {z}''_{\alpha a}{z}'_{\alpha u}  +  4 \lambda  \sum_{\mu = p+1}^{p+q} \frac{\partial^2}{\partial {z}'_{\mu a} \partial {z}''_{\mu u}}, \\
	\omega_\lambda(w_a \circ w_u +iw_a\circ w_ui) &= -4\lambda  \sum_{\alpha=1}^p \frac{\partial^2}{\partial
		{z}''_{\alpha a} \partial {z}'_{\alpha u}} - \frac1{\lambda} \sum_{\mu = p+1}^{p+q} {z}'_{\mu a} {z}''_{\mu u}.
	\end{align*}
\end{itemize}
\end{lemma}

\paragraph{Intertwining}
We now give the intertwiner of  the Fock model for
$\lambda = 2\pi i$ with the Schr\"odinger model in the case when $r=s=n$.
 In that case, the Schr\"odinger model is given by 
the space of Schwartz functions $\Schw(V^n)$ on $V^n$.

The $K'$-finite vectors form the
polynomial Fock space $\Fock(V^n) \subset \Schw(V^n)$ which consists of functions on $V^n$ of the form $p(\mathbf{z})\varphi_0(\mathbf{z})$, where $p(\mathbf{z})$ is
a polynomial function on $V^n$ and $\varphi_0(\mathbf{z})$ is the standard Gaussian on $V^n$. Here we use complex coordinates $\mathbf{z} = (z_1,\dots,z_n)$ with 
$z_i ={^t}(z_{1i},\dots,z_{mi})$ in $V$ relative to the basis
$\{v_{\alpha},v_{\mu}\}$. The Weil representation action of $\mathfrak{sp}(V \otimes W)$ now arises by the following action of the quantum algebra $\W_\lambda$:

\begin{align*}
\omega(v_{\alpha} \otimes w_j'')  & = \sqrt{2} \pi i \left(  \bar{z}_{\alpha j} - \frac{1}{\pi} \frac{\partial}{\partial z_{\alpha j}}  \right), &
\omega(v_{\alpha} \otimes w_{n+j}') &= \sqrt{2} \pi i \left(  \bar{z}_{\alpha j} + \frac{1}{\pi} \frac{\partial}{\partial z_{\alpha j}}  \right)  
, \\
\omega(v_{\alpha} \otimes w_{n+j}'') &= \sqrt{2} \pi i \left( {z}_{\alpha j} - \frac{1}{\pi} \frac{\partial}{\partial \bar{z}_{\alpha j}}  \right)   , &
\omega(v_{\alpha} \otimes w_{j}')&= \sqrt{2} \pi i \left( {z}_{\alpha j} + \frac{1}{\pi} \frac{\partial}{\partial \bar{z}_{\alpha j}}  \right)  ,\\
\omega(v_{\mu} \otimes w_{n+j}')&=-\sqrt{2} \pi i \left(  \bar{z}_{\mu j} - \frac{1}{\pi} \frac{\partial}{\partial z_{\mu j}}
\right)   , &
\omega(v_{\mu} \otimes w_{j}'')&=-\sqrt{2} \pi i \left(  \bar{z}_{\mu j} + \frac{1}{\pi} \frac{\partial}{\partial z_{\mu j}}
\right)  , \\
\omega(v_{\mu} \otimes w_{j}')&=-\sqrt{2} \pi i \left(  z_{\mu j} - \frac{1}{\pi} \frac{\partial}{\partial \bar{z}_{\mu j}}
\right)    , &
\omega(v_{\mu} \otimes w_{n+j}'') &=-\sqrt{2} \pi i \left(  z_{\mu j} + \frac{1}{\pi} \frac{\partial}{\partial \bar{z}_{\mu j}}
\right) .
\end{align*}
Here $1 \leq j \leq n$. We then have a unique $\mathcal{\W}_{\lambda }$-intertwining operator $\iota: S(V^n) \rightarrow \mathcal{P}(\C^{2nm})$ satisfying $\iota (\varphi_0) = 1$, and we have

\begin{lemma}\label{inter1}

	The intertwining operator between the Schr\"odinger and the Fock model satisfies
	\begin{align*}
	\iota \left(  \bar{z}_{\alpha j} - \frac{1}{\pi} \frac{\partial}{\partial z_{\alpha j}}  \right)  \iota^{-1} &= -i \frac1{\sqrt{2}\pi} {z}''_{\alpha j}, &
	\iota \left(  \bar{z}_{\alpha j} + \frac{1}{\pi} \frac{\partial}{\partial z_{\alpha j}}  \right)  \iota^{-1} &
	= 2\sqrt{2}  i\frac{\partial}{\partial {z}'_{\alpha n+j}}, \\
	\iota \left( {z}_{\alpha j} - \frac{1}{\pi} \frac{\partial}{\partial \bar{z}_{\alpha j}}  \right)  \iota^{-1} &= -i \frac1{\sqrt{2}\pi} {z}'_{\alpha n+j}, &
	\iota \left( {z}_{\alpha j} + \frac{1}{\pi} \frac{\partial}{\partial \bar{z}_{\alpha j}}  \right)  \iota^{-1} &= 2\sqrt{2}  i\frac{\partial}{\partial {z}''_{\alpha j}}, \\
	\iota \left(  \bar{z}_{\mu j} - \frac{1}{\pi} \frac{\partial}{\partial z_{\mu j}}
	\right)  \iota^{-1} &= i \frac{1}{\sqrt{2}\pi} {z}''_{\mu n+j} , &
	\iota \left(  \bar{z}_{\mu j} + \frac{1}{\pi} \frac{\partial}{\partial z_{\mu j}}
	\right)  \iota^{-1} &= -2\sqrt{2}  i\frac{\partial}{\partial {z}'_{\mu j}}, \\
	\iota \left(  z_{\mu j} - \frac{1}{\pi} \frac{\partial}{\partial \bar{z}_{\mu j}}
	\right)  \iota^{-1} &= i \frac{1}{\sqrt{2}\pi} {z}'_{\mu j} , &
	\iota \left(  z_{\mu j} + \frac{1}{\pi} \frac{\partial}{\partial \bar{z}_{\mu j}}
	\right)  \iota^{-1} &= -2\sqrt{2}  i\frac{\partial}{\partial {z}''_{\mu n+j}}.
	\end{align*}
\end{lemma}

\bibliographystyle{plain}
\bibliography{unisingular_bib}

\begin{thebibliography}{10}

\bibitem{A07}
Jeffrey Adams.
\newblock The theta correspondence over {$\mathbb{R}$}.
\newblock In {\em Harmonic analysis, group representations, automorphic forms
  and invariant theory}, volume~12 of {\em Lect. Notes Ser. Inst. Math. Sci.
  Natl. Univ. Singap.}, pages 1--39. World Sci. Publ., Hackensack, NJ, 2007.

\bibitem{Bo98}
Richard~E. Borcherds.
\newblock Automorphic forms with singularities on {G}rassmannians.
\newblock {\em Invent. Math.}, 132(3):491--562, 1998.

\bibitem{BES18}
Kathrin {Bringmann}, Stephan {Ehlen}, and Markus {Schwagenscheidt}.
\newblock On the modular completion of certain generating functions.
\newblock {\em ArXiv e-prints}, 2018.

\bibitem{Br02}
Jan~H. Bruinier.
\newblock {\em Borcherds products on {O}(2, {$l$}) and {C}hern classes of
  {H}eegner divisors}, volume 1780 of {\em Lecture Notes in Mathematics}.
\newblock Springer-Verlag, Berlin, 2002.

\bibitem{BrHilbert}
Jan~Hendrik {Bruinier}.
\newblock {Regularized theta lifts for orthogonal groups over totally real
  fields.}
\newblock {\em {J. Reine Angew. Math.}}, 672:177--222, 2012.

\bibitem{BrF04}
Jan~Hendrik {Bruinier} and Jens {Funke}.
\newblock {On two geometric theta lifts.}
\newblock {\em {Duke Math. J.}}, 125(1):45--90, 2004.

\bibitem{BHKRY17I}
Jan-Hendrik {Bruinier}, Ben {Howard}, Stephen~S. {Kudla}, Michael {Rapoport},
  and Tonghai {Yang}.
\newblock Modularity of generating series of divisors on unitary shimura
  varieties.
\newblock preprint, 2017.

\bibitem{BHKRY17II}
Jan-Hendrik {Bruinier}, Ben {Howard}, Stephen~S. {Kudla}, Michael {Rapoport},
  and Tonghai {Yang}.
\newblock Modularity of generating series of divisors on unitary shimura
  varieties {II}: arithmetic applications.
\newblock preprint, 2017.

\bibitem{BHY15}
Jan~Hendrik Bruinier, Benjamin Howard, and Tonghai Yang.
\newblock Heights of {K}udla-{R}apoport divisors and derivatives of
  {$L$}-functions.
\newblock {\em Invent. Math.}, 201(1):1--95, 2015.

\bibitem{BrK03}
Jan~Hendrik Bruinier and Ulf K\"{u}hn.
\newblock Integrals of automorphic {G}reen's functions associated to {H}eegner
  divisors.
\newblock {\em Int. Math. Res. Not.}, (31):1687--1729, 2003.

\bibitem{BY18}
Jan~Hendrik Bruinier and Tonghai Yang.
\newblock Arithmetic degrees of special cycles and derivatives of siegel
  eisenstein series.
\newblock preprint, 2018.

\bibitem{ES16}
Stephan Ehlen and Siddarth Sankaran.
\newblock On two arithmetic theta lifts.
\newblock {\em Compos. Math.}, 154(10):2090--2149, 2018.

\bibitem{EMOT54}
A.~Erd{\'e}lyi, W.~Magnus, F.~Oberhettinger, and F.~G. Tricomi.
\newblock {\em Tables of integral transforms. {V}ol. {I}}.
\newblock McGraw-Hill Book Company, Inc., New York-Toronto-London, 1954.
\newblock Based, in part, on notes left by Harry Bateman.

\bibitem{F80}
Mogens Flensted-Jensen.
\newblock Discrete series for semisimple symmetric spaces.
\newblock {\em Ann. of Math. (2)}, 111(2):253--311, 1980.

\bibitem{FK17}
Jens Funke and Stephen~S Kudla.
\newblock Mock modular forms and geometric theta functions for indefinite
  quadratic forms.
\newblock {\em Journal of Physics A: Mathematical and Theoretical},
  50(40):404001, sep 2017.

\bibitem{FMcoeff}
Jens Funke and John Millson.
\newblock Cycles with local coefficients for orthogonal groups and
  vector-valued {S}iegel modular forms.
\newblock {\em Amer. J. Math.}, 128(4):899--948, 2006.

\bibitem{GS}
Luis {Garcia} and Siddharth {Sankaran}.
\newblock Green forms and the arithmetic siegel-weil formula.
\newblock {\em {I}nv. {M}ath.}
\newblock to appear.

\bibitem{HM96}
Jeffrey~A. Harvey and Gregory Moore.
\newblock Algebras, {BPS} states, and strings.
\newblock {\em Nuclear Phys. B}, 463(2-3):315--368, 1996.

\bibitem{HDiss}
Eric Hofmann.
\newblock {B}orcherds products on unitary groups.
\newblock {\em Mathematische Annalen}, 354:799--832, 2014.

\bibitem{Huf17}
Tobias Hufler.
\newblock {\em {A}utomorphe {F}ormen auf orthogonalen und unit\"{a}ren
  {G}ruppen}.
\newblock PhD thesis, TU Darmstadt, 2017.

\bibitem{K97}
Stephen~S. Kudla.
\newblock Central derivatives of {E}isenstein series and height pairings.
\newblock {\em Ann. of Math. (2)}, 146(3):545--646, 1997.

\bibitem{K03}
Stephen~S. Kudla.
\newblock Integrals of {B}orcherds forms.
\newblock {\em Compositio Math.}, 137(3):293--349, 2003.

\bibitem{K04}
Stephen~S. Kudla.
\newblock Special cycles and derivatives of {E}isenstein series.
\newblock In {\em Heegner points and {R}ankin {$L$}-series}, volume~49 of {\em
  Math. Sci. Res. Inst. Publ.}, pages 243--270. Cambridge Univ. Press,
  Cambridge, 2004.

\bibitem{KM86}
Stephen~S. Kudla and John~J. Millson.
\newblock The theta correspondence and harmonic forms. {I}.
\newblock {\em Math. Ann.}, 274(3):353--378, 1986.

\bibitem{KM87}
Stephen~S. Kudla and John~J. Millson.
\newblock The theta correspondence and harmonic forms. {II}.
\newblock {\em Math. Ann.}, 277(2):267--314, 1987.

\bibitem{KM90}
Stephen~S. Kudla and John~J. Millson.
\newblock Intersection numbers of cycles on locally symmetric spaces and
  {F}ourier coefficients of holomorphic modular forms in several complex
  variables.
\newblock {\em Inst. Hautes \'Etudes Sci. Publ. Math.}, (71):121--172, 1990.

\bibitem{Liu11}
Yifeng Liu.
\newblock {Arithmetic theta lifting and $L$-derivatives for unitary groups. I.}
\newblock {\em Algebra Number Theory}, 5(7):849--921, 2011.

\bibitem{OT09}
Takayuki Oda and Masao Tsuzuki.
\newblock The secondary spherical functions and automorphic {G}reen currents
  for certain symmetric pairs.
\newblock {\em Pure Appl. Math. Q.}, 5(3, Special Issue: In honor of Friedrich
  Hirzebruch. Part 2):977--1028, 2009.

\bibitem{Shin75}
Takuro Shintani.
\newblock On construction of holomorphic cusp forms of half integral weight.
\newblock {\em Nagoya Math. J.}, 58:83--126, 1975.

\end{thebibliography}

\pagebreak[3]

\noindent\textsc{\small
  Department of Mathematical Sciences, University of Durham, South Road, Durham DH1 3LE, UK} \\ 

\noindent\textit{E-mail address:} \href{mailto:jens.funke@durham.ac.uk}{\nolinkurl{jens.funke@durham.ac.uk}}
\\

\noindent\textsc{\small
	Mathematisches Institut, Universit\"{a}t Heidelberg,  
	Im Neuenheimer Feld 205, \newline{D-69120} Heidelberg, Germany}
      \\
      
\noindent\textit{E-mail address:}  
\href{mailto:hofmann@mathi.uni-heidelberg.de}{\nolinkurl{hofmann@mathi.uni-heidelberg.de}} \\

\end{document}